\documentclass{amsart}

\usepackage[foot]{amsaddr}
\let\oldtocsection=\tocsection

\let\oldtocsubsection=\tocsubsection 

\let\oldtocsubsubsection=\tocsubsubsection

\renewcommand{\tocsection}[2]{\vspace{0.5em}\hspace{0em}\oldtocsection{#1}{#2}}
\renewcommand{\tocsubsection}[2]{\vspace{0.5em}\hspace{1em}\oldtocsubsection{#1}{#2}}
\renewcommand{\tocsubsubsection}[2]{\vspace{0.5em}\hspace{2em}\oldtocsubsubsection{#1}{#2}}
\usepackage{graphicx,cancel,xcolor,hyperref,comment,graphicx,geometry}

\usepackage{tikz}
\usepackage{graphicx,url,etoolbox}
\usepackage{lipsum}
\usepackage{dsfont}
\usepackage{amsmath}
\usepackage{amssymb}

\usepackage{fancyhdr}
\usepackage{lastpage}

\newtheorem{theoreme}{Theorem}[section]

\theoremstyle{definition}

\numberwithin{equation}{section}

\renewenvironment{proof}{{\bfseries \noindent Proof.}}{\demo}
\newcommand\xqed[1]{%
	\leavevmode\unskip\penalty9999 \hbox{}\nobreak\hfill
	\quad\hbox{#1}}
\newcommand\demo{\xqed{$\square$}}

\hypersetup{bookmarks, colorlinks, urlcolor=blue, citecolor=blue, linkcolor=blue, hyperfigures, pagebackref,
	pdfcreator=LaTeX, breaklinks=true, pdfpagelayout=SinglePage, bookmarksopen=true,bookmarksopenlevel=2}

\usepackage{comment}

\def\R{\mathbb R}

\def\HH{\mathcal H}

\def\AA{\mathcal A}

\def\la {{\lambda}}

\newcommand {\nc}   {\newcommand}
\nc {\be}   {\begin{equation}} \nc {\ee}   {\end{equation}} \nc
{\beq}  {\begin{eqnarray}} \nc {\eeq}  {\end{eqnarray}} \nc {\beqs}
{\begin{eqnarray*}} \nc {\eeqs} {\end{eqnarray*}}
\def\edc{\end{document}}

\providecommand{\abs}[1]{\lvert#1\rvert}

\numberwithin{equation}{section}

\theoremstyle{Thm}
\newtheorem{Thm}{Theorem}[section]
\newtheorem{lem}{Lemma}[section]
\newtheorem{prop}{Proposition}[section]

\newtheorem{rk}{Remark}[section]
\definecolor{carnelian}{rgb}{0.7, 0.11, 0.11}
\definecolor{carmine}{rgb}{0.59, 0.0, 0.09}
\definecolor{burgundy}{rgb}{0.5, 0.0, 0.13}
\definecolor{darkmidnightblue}{rgb}{0.0, 0.2, 0.4}
\definecolor{dimgray}{rgb}{0.75, 0.75, 0.75}
\definecolor{palecarmine}{rgb}{0.69, 0.25, 0.21}
\newcounter{dummy} 
\numberwithin{dummy}{section}
\newtheorem{Theorem}[dummy]{Theorem}

\newtheorem{defi}[dummy]{Definition}

\numberwithin{equation}{section}

\def\AA{\mathcal A}
\def\HH{\mathbf{\mathcal H}}

\newcommand{\intdx}{\int_{0}^{L} }

\newcommand{\intdnb}{\int_{0 }^{\beta } }

\newcommand{\intdnc}{\int_{\alpha }^{\gamma} }
\newcommand{\Sbt}{S_{\tilde{b}(\cdot)}(u,\omega)}
\newcommand{\Sbtz}{S_{\widetilde{b_0}}(u,\omega)}

\newcommand{\intdm}{\int_{0 }^{\infty } }

\providecommand{\abs}[1]{\lvert#1\rvert}

\begin{document}

		\title[\fontsize{7}{9}\selectfont  ]{Stability results of coupled  wave models with locally memory in a past history framework  via non smooth coefficients on the interface}
	\author{Mohammad Akil$^{1}$}
	\author{Haidar Badawi$^{2}$}
	\author{Serge Nicaise$^{2}$}
	\author{Ali Wehbe$^{3}$}
\address{$^1$ Universit\'e Savoie Mont Blanc, Laboratoire LAMA, Chamb\'ery-France}
\address{$^2$ Universit\'e Polytechnique Hauts-de-France (UPHF-LAMAV),
	Valenciennes, France}
\address{$^3$Lebanese University, Faculty of sciences 1, Khawarizmi Laboratory of  Mathematics and Applications-KALMA, Hadath-Beirut, Lebanon.}
	\email{mohammad.akil@univ-smb.fr, Haidar.Badawi@etu.uphf.fr, Serge.Nicaise@uphf.fr,   ali.wehbe@ul.edu.lb}
	\begin{abstract}
		In this paper, we investigate the stabilization of a locally coupled wave equations with local viscoelastic damping of past history type acting only in one equation via non smooth coefficients. First, using a general criteria of Arendt-Batty,  we prove the strong stability of our system. Second, using a frequency domain approach combined with the multiplier method, we establish the exponential stability of the solution  if and only if the two waves have the same speed of propagation. In case of different speed propagation, we prove that the energy of our system  decays polynomially with rate $t^{-1}$. Finally, we show the lack of exponential stability if the speeds of wave propagation are different.
	\end{abstract}
	\maketitle
	\pagenumbering{roman}
	\maketitle
	\tableofcontents
	\pagenumbering{arabic}
	\setcounter{page}{1}
	\newpage
	

\section{Introduction}
\noindent In this paper, we investigate the indirect stability of coupled elastic wave equations with localized past history damping. More precisely, we consider the following system: 
\begin{equation}\label{paper2-sysorig} \left\{	\begin{array}{llll}\vspace{0.15cm}
\displaystyle u_{tt}-\left(au_x -b(x)\int_{0}^{\infty}g(s)u_x (x,t-s)ds\right)_x +c(x)y_t =0,& (x,s,t)\in (0,L)\times (0,\infty)\times (0,\infty) ,&\\ \vspace{0.15cm}
y_{tt}-y_{xx}-c(x)u_t =0,  &(x,t)\in (0,L)\times (0,\infty) ,&\\\vspace{0.15cm}
u(0,t)=u(L,t)=y(0,t)=y(L,t)=0,& t>0 ,& \\\vspace{0.15cm}
(u(x,-s),u_t (x,0))=(u_0 (x,s),u_1 (x)), &(x,s)\in (0,L)\times(0,\infty),&\\\vspace{0.15cm}	(y(x,0),y_t (x,0))=(y_0 (x),y_1 (x)), &x\in (0,L),
\end{array}\right.
\end{equation}
where $L$ and $a$ are positive real numbers. We suppose that there exists $0<\alpha <\beta <\gamma <L$ and  positive constants $b_0 $ and $c_0$, such that   
\begin{equation}\tag{$b(\cdot)$}\label{p2-b}
b(x)=\left\{	\begin{array}{lll}\vspace{0.15cm}
b_0 ,& x\in (0 ,\beta), &\\
0,&x\in (\beta ,L),&
\end{array}\right. \end{equation}
\begin{equation}\tag{$c(\cdot)$}\label{p2-c}
c(x)=\left\{	\begin{array}{lll}\vspace{0.15cm}
c_0 ,& x\in (\alpha ,\gamma), &\\
0,&x\in (0,\alpha )\cup (\gamma ,L),&
\end{array}\right.
\end{equation}

\begin{figure}[h]
	\begin{center}
		\begin{tikzpicture}
		\draw[->](0,0)--(5,0);
		\draw[->](0,0)--(0,3);
		
		 	\draw[dashed](1,0)--(1,2);
		\draw[dashed](3,0)--(3,2);
		\draw[dashed](2,0)--(2,1);
		
		\node[black,below] at (1,0){\scalebox{0.75}{$\alpha$}};
		\node at (1,0) [circle, scale=0.3, draw=black!80,fill=black!80] {};
		
		\node[black,below] at (2,0){\scalebox{0.75}{$\beta$}};
		\node at (2,0) [circle, scale=0.3, draw=black!80,fill=black!80] {};
		
		\node[black,below] at (3,0){\scalebox{0.75}{$\gamma $}};
		\node at (3,0) [circle, scale=0.3, draw=black!80,fill=black!80] {};

		\node[black,below] at (4,0){\scalebox{0.75}{$L $}};
		\node at (4,0) [circle, scale=0.3, draw=black!80,fill=black!80] {};

		
		\node[black,below] at (0,0){\scalebox{0.75}{$0$}};
		\node at (0,0) [circle, scale=0.3, draw=black!80,fill=black!80] {};
		
			\node at (0,1) [circle, scale=0.3, draw=black!80,fill=black!80] {};
		\node[black,left] at (0,1){\scalebox{0.75}{$b_0$}};
		
		\node at (0,2) [circle, scale=0.3, draw=black!80,fill=black!80] {};
		\node[black,left] at (0,2){\scalebox{0.75}{$c_0$}};

		\node[black,right] at (6.5,3){\scalebox{0.75}{$b(x)$}};
		\node[black,right] at (6.5,2.5){\scalebox{0.75}{$c(x)$}};
	
	\draw[-,red](6,3)--(6.5,3);
	\draw[-,blue](6,2.5)--(6.5,2.5);
		\draw[-,red](0,1)--(2,1);
	    \draw[-,red](2,0)--(4,0);
	
	    \draw[-,blue](0,0)--(1,0);
	    \draw[-,blue](1,2)--(3,2);
	    \draw[-,blue](3,0.011)--(4,0.011);
	
		\end{tikzpicture}\end{center}
	\caption{Geometric description of the functions $b(x)$ and $c(x)$.}\label{p2-Fig0}
\end{figure}
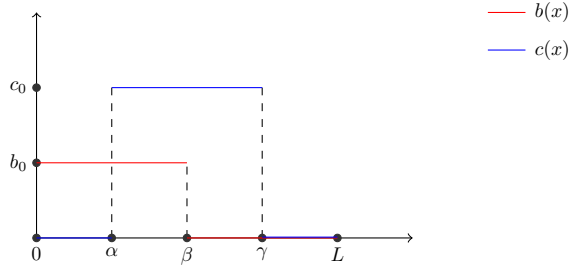
 \noindent the general integral term represents a history term with the relaxation function $g$ that is supposed to satisfy the following hypotheses:
\begin{equation}\tag{${\rm H }$}\label{paper2-H}
\left\{ \begin{array}{lll}
g \in L^1 ([0,\infty))\cap C^1 ([0,\infty)) \ \text{is a  positive function such that }\vspace{0.15cm} \\\displaystyle 
  g(0):=g_0 >0, \ \
 \int_{0}^{\infty}g(s)ds:=\widetilde{g},\ \  
 \widetilde{b}(x):=a-b(x)\widetilde{g}>0, \ \  \text{and}\vspace{0.15cm}\\ 
  g^{\prime}(s)\leq -mg(s), \ \ \text{for some $m>0$}, \forall s \geq 0.
\end{array}\right.
\end{equation}
Remark that, the last assumption in \eqref{paper2-H} implies that 
\begin{equation}\label{2ast}
	g(s)\leq g_{0} e^{-ms}, \ \forall s\geq 0.
\end{equation}
Moreover, from the definition of \ref{p2-b}, we have 
\begin{equation}\tag{$\widetilde{b}(\cdot)$}\label{p2-btilde}
	\widetilde{b}(x):=a-b(x)\widetilde{g}=\left\{\begin{array}{lll}
	\widetilde{b_0}:=a-b_0 \widetilde{g}, & x\in (0,\beta),&\vspace{0.15cm}\\ a,& x\in  (\beta,L).&
	\end{array}\right.
\end{equation}
$\newline$
 The notion of indirect damping mechanisms has been introduced by Russell in \cite{Russell01} and since this time, it retains the attention of many authors. In particular, the fact that only one equation of the coupled system is damped refers to the so-called class of "indirect" stabilization problems initiated and studied in \cite{Alabau04,Alabau02,Alabau05} and further studied by many authors, see for instance  \cite{Alabau06,LiuRao01,ZhangZuazua01} and the rich references therein. In 2008, Rivera {\it et al.} in \cite{MUNOZRIVERA2008482} studied the stability of $1$-dimensional Timoshenko system with past history acting only in one equation, they showed that the system is exponential stable if and only if the equations have the same wave speeds of propagation. In case that the wave speeds of the equations are different, they proved that the solution of the system decays polynomially to zero.  In 2012, Matos {\it et al.}  in \cite{Dilberto} studied the stability of the abstract coupled wave equations with past history, by considering:
\begin{equation}\label{p2-1.2}
\left\{	\begin{array}{lll}
\displaystyle	u_{tt}+\mathbb{A}_1u-\intdm g(s)\mathbb{A}_2 u (t-s)ds+\beta v =0,\vspace{0.15cm}\\
	v_{tt}+\mathbb{B}v+\beta u=0, \ \ \text{in}\ \ L^2 (\R^{+},\HH),\vspace{0.15cm}\\
	u(-t)=u_0(t), \ \ t\geq 0, \vspace{0.15cm}\\
	v(0)=v_0, \vspace{0.15cm}\\
	u_t(0)=u_1 , \ \ v_t (0)=v_1, 
	\end{array}
	\right.
\end{equation}
where $\mathbb{A}_1$, $\mathbb{A}_2$ and $\mathbb{B}$  are self-adjoint positive-definite operators with the domain $D(\mathbb{A}_1)\subseteq D(\mathbb{A}_2)\subset \HH $ and $D(\mathbb{B})\subset \HH$ with compact embeddings in $\HH$, $g:[0,\infty)\longmapsto [0,\infty)$ is a smooth and summable function and $\beta $ is a small positive constant. They showed that the abstract setting is not strong enough to produce exponential stability and they proved that the solution decays polynomially to zero.
In 2014, Fatori {\it et al.} in \cite{Fatoriluci} studied a fully hyperbolic thermoelastic Timoshenko system with past history where the thermal effects are given by Cattaneo's law, they established the exponential stability of the solution if and only if  the coefficients of their System satisfy the next relation
$\chi_0 :=\left(\tau-\frac{\rho_1}{\rho_3 \kappa}\right)\left(\rho_2 -\frac{b\rho_1}{\kappa}\right)-\frac{\tau\rho_1 \delta^2}{\rho_3 \kappa} =0$. In the case  $\chi_0 \neq 0$, they established optimal polynomial stability rates. In the same year, Santos {\it et al.} in \cite{Santossoufyanejunior}  studied the stability of $1$-dimensional Bresse system with past history acting in the shear angle displacement, they showed the exponential decay of the solution if and only if the wave speeds are the same. Otherwise, they showed that the Bresse system is polynomial stable with optimal decay rate.
In 2014, Alabau-Boussouira {\it et al.} in \cite{FNP} studied the behaviour of the wave equation with viscoelastic damping in the presence of time-delayed damping, by considering:
\begin{equation}\label{p2-1.3}
\left\{	\begin{array}{lll}
\displaystyle 	u_{tt}(x,t)-\Delta u(x,t)+\intdm \mu (s)\Delta u(x,t-s)ds+ku_t (x,t-\tau)=0, \ \ \text{in}\ \ \Omega \times (0,\infty),\vspace{0.15cm}\\
	u(x,t)=0, \ \ \text{in}\ \ \partial \Omega \times (0,\infty),\vspace{0.15cm}\\
	u(x,t)=u_0 (x,t) \ \ \text{in}\  \ \Omega \times (-\infty,0],
	\end{array}
	\right.
\end{equation}
where $\Omega \subset \R^n $ is an open bounded set with a smooth boundary $\partial \Omega$, the initial data $u_0$ belongs to a suitable space, the constant $\tau >0$ is the time delay, $k$ is a real number and the memory kernel $\mu:[0,\infty)\longmapsto [0,\infty)$ is a locally absolutely continuous function satisfying $\mu (0)=\mu_0 >0$, $\intdm \mu(t)dt=\tilde{\mu}<1$ and $\mu^{\prime}(t)\leq -\alpha \mu (t)$, for some $\alpha >0$. They proved an exponential stability result provided that the amplitude $k$ of time-delayed damping is small enough. Also, they showed that even if the delay effect usually generates instabilities, the damping due to viscoelastic can counterbalance them. Moreover, in \cite{Giorgi} they established an exponential stability of the system \eqref{p2-1.3} with $k=0$.
In 2018, Abdallah, Ghader and Wehbe in \cite{Wehbe08} studied the stability of a $1$-dimensional Bresse system with infinite memory type control and /or with heat conduction given by Cattaneo's law acting in the shear angle displacement. In the absence of thermal effect, under the same speed propagation, they established  the exponential stability of the system. However, in the case of different speed propagation, they established a polynomial energy decay rate. In 2018, Cavalcanti {\it et al.} in \cite{CAVALCANTI20186535} studied the asymptotic stability of the  multidimensional damped wave equation, by considering:
\begin{equation}
	\rho(x)u_{tt}-\Delta u +\intdm g(s)\text{div}[a(x)\nabla u (\cdot,t-s)]ds+b(x)u_t =0, \ \ \text{in}\ \ \Omega\times (0,\infty),
\end{equation}
where $\Omega$ is an open bounded and connected set of $\R^n$, $n\geq 2$, $\rho (x)$ is constant,  $a(x)\geq 0$ is a smooth function, $b(x)\geq 0$ is a bounded function acting effectively  in a region $A$ of $\Omega$ where $a=0$. Considering that the well-known geometric control condition $(\omega,T_0)$ holds and supposing that the relaxation function $g $ is bounded by a function that decays exponentially to zero, they proved that the solution to the corresponding partial viscoelastic model decays exponentially to zero, even in the absence of the frictional dissipative effect. Moreover, they proved by removing the frictional damping term $b(x)u_t$ and by assuming that $\rho$ is not constant, that localized viscoelastic damping is strong enough to assure that the system is exponentially stable. In 2011, Almeida {\it and al.} in \cite{Almeida2011} studied the stability of coupled wave equations with past history effective only in one equation, by considering the following system:
\begin{equation}\label{almeida}
\left\{	\begin{array}{lll}
\displaystyle	u_{tt}-\Delta u +\intdm g(s)\Delta u(\cdot,t-s)ds+\alpha v =0,\ \ \text{in}\ \ \Omega \times (0,\infty),\vspace{0.15cm}\\
	v_{tt}-\Delta v +\alpha u =0,\ \ \text{in}\ \ \Omega \times (0,\infty),\vspace{0.15cm}\\
	u=v=0, \ \ \text{on}\ \ \Gamma \times(0,\infty)\vspace{0.15cm}\\
	u(x,0),v(x,0))=(u_0 (x),v_0(x)) \ \ \text{in}\  \ \Omega,\vspace{0.15cm}\\
		u_t(x,0),v_t(x,0))=(u_1 (x),v_1(x)) \ \ \text{in}\  \ \Omega,
	\end{array}
	\right.
\end{equation}
where $\Omega$ is an open bounded set of $\R^n$ with smooth boundary $\Gamma$ and $\alpha >0$. They showed that the dissipation given by the memory effect is not strong enough to produce exponential decay. They proved that the solution of the system \eqref{almeida} decays polynomially with  rate  $t^{-\frac{1}{2}}$. Also, in 2020, Cordeiro {\it et al.} in \cite{cordeiro} etablished the optimality of the decay rate.
\\[0.1in]
But to the best of our knowledge, it seems that  no result in the literature exists concerning the case of coupled wave equations with localized past history damping, especially in the absence of smoothness of the damping and coupling  coefficients. The goal of the present paper is to fill this gap by studying the stability of system \eqref{paper2-sysorig}.
\\[0.1in]
This paper is organized as follows: In Section \ref{paper2-WPS}, we prove the well-posedness of our system by using semigroup approach. In Section \ref{p2-sec3}, following a general criteria of Arendt Batty, we show the strong stability of our system in the absence of the compactness of the resolvent.  Next, in Section \ref{Exp-Pol}, by using the frequency domain approach  combining with a specific  multiplier method, we establish exponential stability of the solution  if and only if the waves have same speed propagation (i.e. $a=1$). In the case $a\neq 1$, we prove that the energy of our system  decays polynomially with the rate $t^{-1}$. Finally, in Section \ref{p2-lackexpostability}, we show the lack of exponential stability in case that the speeds of wave propagation are different, i.e., when  $a\neq 1$.
\section{Well-posedness of the system}\label{paper2-WPS}
\noindent In this section, we will establish the well-posedness of  system \eqref{paper2-sysorig} by using semigroup approach. To this aim, as in \cite{Dafermos01},
we introduce the following auxiliary change of variable

\begin{equation}\label{p2-changeofvar}
\omega(x,s,t):=u(x,t)-u(x,t-s),\ \ (x,s,t)\in (0,L)\times (0,\infty)\times (0,\infty).
\end{equation}
Then, system \eqref{paper2-sysorig} becomes
\begin{eqnarray}	u_{tt}-\left(\widetilde{b}(\cdot)u_x +b(\cdot)\int_{0}^{\infty}g(s)\omega_x (\cdot,s,t) \right)_x  +c(\cdot)y_t =0,& (x,s,t)\in (0,L)\times (0,\infty)\times (0,\infty), &\label{p2-sysorg0}\\ \vspace{0.15cm}
y_{tt}-y_{xx}-c(\cdot)u_t =0,  &(x,t)\in (0,L)\times (0,\infty) ,&\label{p2-sysorg1}\\\vspace{0.15cm}
\omega_{t}(\cdot,s,t)+\omega_{s}(\cdot,s,t)-u_t =0,  &(x,s,t)\in (0,L)\times (0,\infty)\times (0,\infty), &\label{p2-sysorg2}\vspace{0.15cm}
\end{eqnarray}
with the following boundary  conditions 
\begin{equation}\label{p2-bc3}
\left\{\begin{array}{lll}
	u(0,t)=u(L,t)=y(0,t)=y(L,t)=0,\   \ t>0, \vspace{0.15cm}\\ \omega(\cdot,0,t)=0,\ \ (x,t)\in (0,L)\times(0,\infty), \vspace{0.15cm}\\ \omega(0,s,t)=0 , \ \ (s,t)\in (0,\infty)\times(0,\infty),
	\end{array}\right.
	 \end{equation}
and the following initial conditions
\begin{equation}\label{p2-initialcon}
\left\{\begin{array}{llll}
u(\cdot,-s)=u_0 (\cdot,s),\qquad u_t (\cdot,0)=u_1 (\cdot), &(x,s)\in (0,L)\times(0,\infty),& \vspace{0.15cm}\\
y(\cdot,0)=y_0 (\cdot),\qquad y_t (\cdot,0)=y_1 (\cdot), &x\in (0,L),&\vspace{0.15cm} \\
\omega(\cdot,s,0)=u_0 (\cdot,0)-u_0 (\cdot,s), &(x,s)\in (0,L)\times(0,\infty).&\vspace{0.15cm} 
\end{array}\right.\end{equation} 
The energy of  system \eqref{p2-sysorg0}-\eqref{p2-initialcon} is given by 
\begin{equation}\label{p2-Energy}
E(t)=E_1 (t)+E_2 (t)+E_3 (t),
\end{equation}where
\begin{equation*}
E_1 (t)=\frac{1}{2}\int_{0}^{L}\left(\left|u_t \right|^2 +\widetilde{b}(\cdot)|u_x |^2 \right)dx,\ \
E_2 (t)=\frac{1}{2}\int_{0}^{L}\left(\left|y_t \right|^2 +|y_x |^2 \right)dx\ \ \text{and}\ \ 
E_3 (t)=\frac{b_0 }{2}\intdnb \int_{0}^{\infty}g(s)|\omega_x (\cdot,s,t) |^2 dsdx.    \end{equation*} 
\begin{lem}\label{p2-lemenergy}
	{\rm	 Under the hypotheses \eqref{paper2-H}. Let  $U=(u,u_t ,y,y_t ,\omega )$ be a regular solution of  system \eqref{p2-sysorg0}-\eqref{p2-initialcon}. Then, the energy $E(t)$  satisfies the following estimation }
	\begin{equation}\label{p2-dewrtt}
		\frac{d}{dt}E(t)= \frac{b_0}{2}\intdnb \int_{0}^{\infty}g^{\prime}(s)|\omega_x (\cdot,s,t)|^2 dsdx .
		\end{equation}
\end{lem}
\begin{proof}
	First, multiplying  \eqref{p2-sysorg0} by $\overline{u_t }$, integrating over $(0,L)$, using integration by parts with  \eqref{p2-bc3}, using the definition of \ref{p2-b} and  \ref{p2-c}, then taking the real part, we obtain
	\begin{equation}\label{p2-utbar}
	\begin{array}{lll}
	\displaystyle  \frac{d}{dt}E_{1} (t)=\displaystyle-\,\Re\left\{b_0  \intdnb\int_{0}^{\infty}g(s)\omega_x (\cdot,s,t) \overline{u_{tx} }dsdx\right\} - \Re \left\{c_0\intdnc y_t \overline{u_t }dx \right\}.
	\end{array}
	\end{equation}
	Now,  multiplying  \eqref{p2-sysorg1} by $\overline{y_t }$, integrating over $(0,L)$, using the definition of \ref{p2-c}, then taking the real part, we get
	\begin{equation}\label{p2-E2'}
	\frac{d}{dt}E_{2} (t)= \Re\left\{c_0 \intdnc u_t \overline{y_t }dx \right\}.
	\end{equation}Deriving  \eqref{p2-sysorg2} with respect to $x$, we obtain
	\begin{equation}\label{p2-d/dx b(x)}
	\omega_{xt}(\cdot,s,t)+\omega_{xs}(\cdot,s,t)-u_{tx}=0.
	\end{equation}
	Multiplying  \eqref{p2-d/dx b(x)} by $b_0 g(s)\overline{\omega_x }(\cdot,s,t)$, integrating over $(0,\beta)\times (0,\infty)$, then taking the real part, we get	\begin{equation*}
	\frac{d}{dt}E_{3} (t)=-\frac{b_0 }{2}\intdnb\int_{0}^{\infty}g(s)\frac{d}{ds}|\omega_x (\cdot,s,t)|^2 dsdx+\Re\left\{b_0 \intdnb\int_{0}^{\infty}g(s)\overline{\omega_x } (\cdot,s,t)u_{tx}dsdx \right\}.
	\end{equation*}
	Using integration by parts with respect to $s$ in the above equation with the help of \eqref{p2-bc3} and the hypotheses \eqref{paper2-H}, we obtain
	\begin{equation}\label{p2-E3'}
	\frac{d}{dt}E_{3} (t)=\frac{b_0 }{2}\intdnb\int_{0}^{\infty}g^{\prime}(s)|\omega_x (\cdot,s,t)|^2 dsdx+\Re\left\{b_0 \intdnb\int_{0}^{\infty}g(s)\overline{\omega_x } (\cdot,s,t)u_{tx}dsdx \right\}.
	\end{equation}
	Finally,  adding  \eqref{p2-utbar}, \eqref{p2-E2'} and  \eqref{p2-E3'}, we obtain \eqref{p2-dewrtt}. The proof is thus complete.
\end{proof}\\\linebreak	
Under the hypotheses  \eqref{paper2-H} and  from Lemma \ref{p2-lemenergy},  system \eqref{p2-sysorg0}-\eqref{p2-initialcon} is dissipative in the sense that its energy is non-increasing with respect to time (i.e. $E^{\prime}(t)\leq 0$). Now, we define the following Hilbert  space $\HH$ by:
\begin{equation*}
\HH:=\left(H_{0}^1 (0,L)\times L^2 (0,L)\right)^2\times \mathcal{W }_g , 
\end{equation*}where
\begin{equation*}
\quad \mathcal{W}_g :=	L^2_g ((0,\infty) ;H^1_{L} (0 ,\beta ))\quad \text{and} \quad 
H^1_{L} (0,\beta):=\left\{\widetilde{\omega}\in H^1 (0,\beta) \ | \ \widetilde{\omega}(0)=0\right\}.
\end{equation*}
The space 
$\mathcal{W}_g$ is an Hilbert space of $ H^1_{L} (0 ,\beta )$-valued functions on $(0,\infty) $, equipped with the following inner product
\begin{equation*}
(\omega^1 ,\omega^2  )_{\mathcal{W}_g}:=\intdnb\int_{0}^{\infty}g(s)\omega^1 _x  \overline{\omega^2_x }ds dx, \quad \forall \, \omega^1 , \omega^2 \in \mathcal{W}_g .
\end{equation*}
The Hilbert space $\mathcal{H}$ is equipped with the following inner product \begin{equation}\label{p2-inner}\left(U,U^1 \right)_{\HH}=\intdx \left(\widetilde{b}(\cdot)u_x \overline{u_{x}^1 }+v\overline{v^1 }+y_x \overline{y_{x}^1 }+z\overline{z^1 }\right)dx + b_0 \intdnb\int_{0}^{\infty}g(s)\omega_x (\cdot,s) \overline{\omega_{x}^1 } (\cdot,s) ds dx,  \end{equation}where  $U=(u,v,y,z,\omega(\cdot,s))^{\top}\in \HH$ and $U^1 =(u^1 ,v^1 ,y^1 ,z^1 ,\omega^1 (\cdot,s))^{\top}\in\HH$. Now, we define the linear unbounded  operator $\AA:D(\AA)\subset \HH\longmapsto \HH$  by:
\begin{equation}
D(\AA)=\left\{\begin{array}{cc}\vspace{0.25cm}
U=(u,v,y,z,\omega(\cdot,s))^{\top}\in \HH \,\,|\,\, y\in H^2 (0,L)\cap H^{1}_0 (0,L),\,\, v,z\in H_0^1 (0,L) \\\vspace{0.25cm}
\displaystyle  \left(S_{\tilde{b}(\cdot)}(u,\omega)\right)_x  \in L^2 (0,L),\quad \omega_s (\cdot,s)\in \mathcal{W}_g ,\quad \omega(\cdot,0)=0.
\end{array}\right\} 
\end{equation}and 
\begin{equation}\label{p2-op}
\AA\begin{pmatrix}
u\\v\\y\\z\\ \omega(\cdot,s)
\end{pmatrix}=
\begin{pmatrix} 
v\\\displaystyle \left(S_{\tilde{b}(\cdot)}(u,\omega)\right)_x  -c(\cdot)z\\z\\y_{xx} +c(\cdot)v\\-\omega_s (\cdot,s) +v 
\end{pmatrix},
\end{equation} 
where $\displaystyle S_{\tilde{b}(\cdot)}(u,\omega):= \widetilde{b}(\cdot)u_x +b(\cdot)\int_{0}^{\infty}g(s)\omega_x (\cdot,s)ds$.
Moreover, from the definition of \ref{p2-b} and \ref{p2-btilde}, we have 
\begin{equation}
\tag{$S_{\tilde{b}(\cdot)}(u,\omega)$}
\label{p2-S}
 S_{\tilde{b}(\cdot)}(u,\omega) =\left\{	\begin{array}{lll}\vspace{0.15cm}
\displaystyle S_{\widetilde{b_0}}(u,\omega):= \widetilde{b_0}u_x +b_0 \intdm g(s)\omega_x (\cdot,s)ds,& x\in (0 ,\beta), &\\
\displaystyle au_x ,&x\in  (\beta ,L).&
\end{array}\right.
\end{equation}
Now, if $U=(u, u_t ,y, y_t ,\omega(\cdot,s))^{\top}$, then system \eqref{p2-sysorg0}-\eqref{p2-initialcon} can be written as the following first order evolution equation 
\begin{equation}\label{p2-firstevo}
U_t =\AA U , \quad U(0)=U_0,
\end{equation}where  $U_0 =(u_0 (\cdot,0) ,u_1 ,y_0 ,y_1 ,\omega_0 (\cdot,s) )^{\top}\in \HH$.
\begin{prop}\label{p2-mdissip}{\rm
		Under the hypotheses \eqref{paper2-H}, the unbounded linear operator $\AA$ is m-dissipative in the energy space $\HH$.} 
\end{prop}
\begin{proof}
	For all $U=(u,v,y,z,\omega(\cdot,s))^{\top}\in D(\AA)$, from \eqref{p2-inner} and \eqref{p2-op}, we have 
	$$
	\begin{array}{lll}
\displaystyle	\Re(\AA U,U)_{\HH}=\displaystyle\Re\left\{\intdx\widetilde{b}(\cdot)v_x \overline{u_x}dx\right\}+\Re\left\{\intdx\left(S_{\tilde{b}(\cdot)}(u,\omega)\right)_x \overline{v}dx\right\}+\Re\left\{\intdx z_x \overline{y_x}dx\right\}+\Re\left\{\intdx y_{xx} \overline{z}dx\right\}\vspace{0.25cm}\\\hspace{2cm}\displaystyle +\, \Re\left\{b_0 \intdnb\intdm g(s)v_x \overline{\omega_x}(\cdot,s)dsdx \right\}-\Re\left\{b_0 \intdnb\intdm g(s)\omega_{xs}(\cdot,s) \overline{\omega_x}(\cdot,s)dsdx \right\}.
\end{array}
	$$
	Using integration by parts to the second and fourth terms in the above equation, then using the fact that $U\in D(\AA)$ , we obtain 
	$$
	\Re(\AA U,U)_{\HH}=-\Re\left\{b_0 \intdnb\intdm g(s)\omega_{xs}(\cdot,s) \overline{\omega_x}(\cdot,s)dsdx \right\}=-\frac{b_0}{2}\intdnb\intdm g(s)\frac{d}{ds}|\omega_x (\cdot,s)|^2dsdx.
	$$
	Using integration by parts with respect to $s$ in the above equation and the fact that $\omega(\cdot,0)=0$ with the help of hypotheses \eqref{paper2-H}, we get 
	\begin{equation}\label{p2-reauu}
	\begin{array}{lll}
	\displaystyle	\Re \left(\AA U,U\right)_{\HH} =\displaystyle\frac{b_0}{2}\intdnb\int_{0}^{\infty}g^{\prime}(s)|\omega_x (\cdot,s)|^2 dsdx\leq0,\end{array}\end{equation}
	which implies that $\AA$ is dissipative. Now, let us prove that $\AA$ is maximal. For this aim, let $F=(f^1 ,f^2 ,f^3 ,f^4 ,f^5 (\cdot,s) )^{\top}\in \HH$, we want to find $U=(u,v,y,z,\omega(\cdot,s))^{\top}\in D(\AA)$ unique solution  of \begin{equation}\label{p2--au=f}
	-\AA U=F.
	\end{equation}Equivalently, we have the following system
	\begin{eqnarray}
	-v&=&f^1 , \label{p2-f1}\\
	-\left(S_{\tilde{b}(\cdot)}(u,\omega) \right)_x +c(\cdot)z &=&f^2 ,\label{p2-f2}\\
	-z&=&f^3 , \label{p2-f3}\\
	-y_{xx}-c(\cdot)v&=&f^4 , \label{p2-f4}\\
	\omega_s (\cdot,s)-v&=&f^5(\cdot,s) , \label{p2-f5}
	\end{eqnarray}with the following boundary conditions \begin{equation}\label{p2-boundary1mdissip}
	u(0)=u(L)=y(0)=y(L)=0,\ \ \omega (\cdot,0)=0\ \ \text{in}\ \ (0,L) \ \ \text{and} \ \ \omega(0,s)=0  \ \ \text{in} \ \ (0,\infty).
	\end{equation}
	From \eqref{p2-f1}, \eqref{p2-f5} and \eqref{p2-boundary1mdissip}, we get
	\begin{equation}\label{p2-2.25}
	\omega(x,s)=\int_{0}^{s}f^5 (x,\xi)d\xi-sf^1,\ \ (x,s)\in (0,L)\times(0,\infty).
	\end{equation}Since $v=-f^1 \in H^{1}_0 (0,L) $ and $f^5 (\cdot,s)\in \mathcal{W}_g $, then from \eqref{p2-f5} and \eqref{p2-2.25} we get  $\omega_s (\cdot,s)\in \mathcal{W}_g $ and $\omega (\cdot,s) \in H^1_L (0,\beta)$ a.e. in $(0,\infty)$. Now, to obtain that $\omega(\cdot,s) \in \mathcal{W}_g$, it is sufficient to prove that $\displaystyle \intdm g(s)\|\omega_x(\cdot,s)\|_{L^2_{0,\beta}}^2 ds<\infty$ where $\|\cdot\|_{L^2_{0,\beta}}:=\|\cdot\|_{L^2 (0,\beta)}$ . For this aim, let $\epsilon_1,\epsilon_2 >0 $ , under the hypotheses \eqref{paper2-H}, we have 
	\begin{equation}\label{p2-winl2-1}
	\int_{\epsilon_1}^{\epsilon_2}g(s)\|\omega_x(\cdot,s)\|_{L^2_{0,\beta}}^2ds \leq -\frac{1}{m}	\int_{\epsilon_1}^{\epsilon_2}g^{\prime}(s)\|\omega_x(\cdot,s)\|_{L^2_{0,\beta}}^2ds.
	\end{equation}
	Using integration by parts in \eqref{p2-winl2-1}, we obtain
	\begin{equation*}\label{p2-winl2-2}
	\int_{\epsilon_1}^{\epsilon_2}g(s)\|\omega_x(\cdot,s)\|_{L^2_{0,\beta}}^2ds \leq \frac{1}{m}\left[	\int_{\epsilon_1}^{\epsilon_2}g(s)\frac{d}{ds}\left(\|\omega_x(\cdot,s)\|_{L^2_{0,\beta}}^2\right) ds+g(\epsilon_1)\|\omega_x (\cdot,\epsilon_1)\|^2_{L^2_{0,\beta} } -g\left(\epsilon_2\right)\left\|\omega_x \left(\cdot,\epsilon_2\right)\right\|^2_{L^2_{0,\beta} }\right]. 
	\end{equation*}
	Moreover, from Young's inequality, we have 
	\begin{equation}\label{p2-2.27}
	\begin{array}{lll}
	\displaystyle  \frac{1}{m}	\int_{\epsilon_1}^{\epsilon_2}g(s)\frac{d}{ds}\left(\|\omega_x(\cdot,s)\|_{L^2_{0,\beta}}^2\right) ds&=&\displaystyle \frac{2}{m}\int_{\epsilon_1}^{\epsilon_2}g(s)\Re\left\{\intdnb\omega_x (\cdot,s)\overline{\omega_{sx}}(\cdot,s) dx\right\}ds\vspace{0.25cm}\\
	&\leq&	\displaystyle \frac{1}{2}\int_{\epsilon_1}^{\epsilon_2}g(s)\|\omega_x(\cdot,s)\|_{L^2_{0,\beta}}^2ds+\frac{2}{m^2 }\int_{\epsilon_1}^{\epsilon_2}g(s)\|\omega_{sx}(\cdot,s)\|_{L^2_{0,\beta}}^2ds.
	 \end{array}
	\end{equation}
	Inserting \eqref{p2-2.27} in the above inequality, we get 
	\begin{equation*}
		\int_{\epsilon_1}^{\epsilon_2}g(s)\|\omega_x(\cdot,s)\|_{L^2_{0,\beta}}^2ds \leq \frac{4}{m^2}	\int_{\epsilon_1}^{\epsilon_2}g(s)\|\omega_{sx}(\cdot,s)\|_{L^2_{0,\beta}}^2 ds+\frac{2}{m}g(\epsilon_1)\|\omega_x (\cdot,\epsilon_1)\|^2_{L^2_{0,\beta} } -\frac{2}{m}g\left(\epsilon_2\right)\left\|\omega_x \left(\cdot,\epsilon_2\right)\right\|^2_{L^2_{0,\beta} 
		 }. 
	\end{equation*}
	Using the fact that $\omega_s(\cdot,s)\in \mathcal{W}_g$, $\omega(\cdot,0)=0$ and the hypotheses \eqref{paper2-H} in the above inequality, (in particular \eqref{2ast}) we obtain, as $\epsilon_1 \to 0^+$ and $\epsilon_2\to \infty$, that 
		\begin{equation*}
	\int_{0}^{\infty}g(s)\|\omega_x(\cdot,s)\|_{L^2_{0,\beta}}^2ds <\infty,
	\end{equation*}
	and consequently, $\omega(\cdot,s)\in \mathcal{W}_g $.
	Now, see the definition of \ref{p2-S}, substituting \eqref{p2-f1}, \eqref{p2-f3} and \eqref{p2-2.25} in \eqref{p2-f2} and \eqref{p2-f4}, we get the following system 
	\begin{eqnarray}
	\left[\widetilde{b}(\cdot)u_x +b(\cdot)\left(\int_{0}^{\infty}g(s)\left(\int_{0}^s f^{5}_x (\cdot,\xi)d\xi-sf^{1}_x \right)ds\right)\right]_x +c(\cdot)f^3 =-\,f^2, \label{ulax}\\ y_{xx}-c(\cdot)f^1=-\,f^4,\label{ylax}\\\vspace{0.25cm}
	u(0)=u(L)=y(0)=y(L)=0. \label{bculaxylax}
	\end{eqnarray}
	Let $(\phi ,\psi) \in H^{1}_0 (0,L)  \times H^{1}_0 (0,L)$. Multiplying  \eqref{ulax} and \eqref{ylax} by $\overline{\phi}$ and $\overline{\psi}$ respectively, integrating over $(0,L)$, then using formal integrations by parts, we obtain
\begin{equation}\label{ulax1}
		\begin{array}{lll}
		\displaystyle	\intdx \widetilde{b}(\cdot)u_x \overline{\phi_x }dx=	\displaystyle \intdx f^2 \overline{\phi}dx+c_0 \intdnc f^3 \overline{\phi}dx-b_0\intdnb\int_{0}^{\infty}g(s)\left(\int_{0}^{s}f^5_x (\cdot,\xi)d\xi-sf^{1}_x \right)\overline{\phi_x }dsdx
		\end{array}
		\end{equation}
	and 
	\begin{equation}\label{ylax2}
	\intdx y_x \overline{\psi_x }dx=\intdx f^4 \overline{\psi}dx-c_0 \intdnc f^1 \overline{\psi}dx.
	\end{equation} 
	Adding  \eqref{ulax1} and \eqref{ylax2}, we obtain
	\begin{equation}\label{vf}
	\mathcal{B}((u,y),(\phi,\psi))=\mathcal{L}(\phi,\psi), \quad\forall (\phi,\psi)\in H^{1}_0 (0,L)\times  H^{1}_0 (0,L), 
	\end{equation} where $$\mathcal{B}((u,y),(\phi,\psi))=\displaystyle\intdx \widetilde{b}(\cdot) u_x \overline{\phi_x }dx+\intdx y_x \overline{\psi_x }dx$$ and $$\begin{array}{lll} \displaystyle \mathcal{L}(\phi,\psi)=	\displaystyle \intdx \left(f^2 \overline{\phi}+f^4 \overline{\psi}\right)dx+c_0 \intdnc \left(f^3 \overline{\phi}-f^1 \overline{\psi}\right)dx-b_0\intdnb\int_{0}^{\infty}g(s)\left(\int_{0}^{s}f^5_x (\cdot,\xi)d\xi-sf^{1}_x \right)\overline{\phi_x }dsdx.  \end{array}$$
	It is easy to see that,  $\mathcal{B}$ is a sesquilinear, continuous and coercive form on $\left( H^{1}_0 (0,L)  \times H^{1}_0 (0,L)\right)^2 $ and $\mathcal{L}$ is a linear and continuous form on $ H^{1}_0 (0,L) \times H^{1}_0 (0,L)$. Then, it follows by Lax-Milgram theorem that \eqref{vf} admits a unique solution $(u,y)\in H^{1}_0 (0,L)\times H^{1}_0 (0,L) $. By using the classical elliptic regularity, we deduce that the system \eqref{ulax}-\eqref{bculaxylax} admits a unique solution  $(u,y)\in H^{1}_0 (0,L)\times \left(H^2 (0,L)\cap H^{1}_0 (0,L)\right) $ such that $(S_{\tilde{b}(\cdot)}(u,\omega))_x\in L^2 (0,L)$ and consequently,  $U \in D(\AA) $ is a unique solution of \eqref{p2--au=f}. Then, $\mathcal{A}$ is an isomorphism and since $\rho\left(\mathcal{A}\right)$ is open set of $\mathbb{C}$ (see Theorem 6.7 (Chapter III) in \cite{Kato01}),  we easily get $R(\lambda I -\mathcal{A}) = {\mathcal{H}}$ for a sufficiently small $\lambda>0 $. This, together with the dissipativeness of $\mathcal{A}$, imply that   $D\left(\mathcal{A}\right)$ is dense in ${\mathcal{H}}$   and that $\mathcal{A}$ is m-dissipative in ${\mathcal{H}}$ (see Theorems 4.5, 4.6 in  \cite{Pazy01}). The proof is thus complete.
\end{proof}\\\linebreak
According to Lumer-Philips theorem (see \cite{Pazy01}), Proposition \ref{p2-mdissip} implies that the operator $\AA$ generates a $C_{0}$-semigroup of contractions $e^{t\AA}$ in $\HH$ which gives the well-posedness of \eqref{p2-firstevo}. Then, we have the following result:
\begin{Thm}{\rm
		Under the hypotheses \eqref{paper2-H}, for all $U_0 \in \HH$,  System \eqref{p2-firstevo} admits a unique weak solution $$U(x,s,t)=e^{t\AA}U_0(x,s) \in C^0 (\R^+ ,\HH).
		$$ Moreover, if $U_0 \in D(\AA)$, then the system \eqref{p2-firstevo} admits a unique strong solution $$U(x,s,t)=e^{t\AA}U_0(x,s) \in C^0 (\R^+ ,D(\AA))\cap C^1 (\R^+ ,\HH).$$}
\end{Thm}
\section{Strong Stability}\label{p2-sec3}
\noindent This section is devoted to the proof of the strong stability of the $C_0$-semigroup $\left(e^{t\mathcal{A}}\right)_{t\geq0}$.
\noindent To obtain the strong stability of the $C_0$-semigroup $\left(e^{t\mathcal{A}}\right)_{t\geq0}$, we use the theorem of Arendt and Batty in \cite{Arendt01} (see Theorem \ref{App-Theorem-A.2} in Appendix \ref{p2-appendix}).

\begin{theoreme}\label{p2-strongthm2}
{\rm	Assume that the hypotheses \eqref{paper2-H} hold. Then, the $C_0-$semigroup of contraction $\left(e^{t\AA}\right)_{t\geq 0}$ is strongly stable in $\HH$; i.e., for all $U_0\in \HH$, the solution of \eqref{p2-firstevo} satisfies 
	$$
	\lim_{t\rightarrow +\infty}\|e^{t\AA}U_0\|_{\HH}=0.
	$$}
\end{theoreme}\noindent According to Theorem \ref{App-Theorem-A.2}, to prove Theorem \ref{p2-strongthm2}, we need to prove that the operator $\AA$ has no pure imaginary eigenvalues and $\sigma(\AA)\cap i\R $ is countable. The proof of Theorem \ref{p2-strongthm2} has been divided into the following two Lemmas.
\begin{lem}\label{p2-ker}
	{\rm Under the hypotheeis \eqref{paper2-H}, we have 
		$$\ker(i\la I-\AA)=\{0\},\ \  \forall \la \in \R.$$

	}
\end{lem}
\begin{proof}
	From Proposition \ref{p2-mdissip}, we have $0\in\rho(\AA)$. We still need to show the result for $\la \in \R^{\star}$. For this aim, suppose that there exists a real number $\la\neq0$ and $U=(u,v,y,z,\omega(\cdot,s))^{\top}\in D(\AA)$ such that 
	\begin{equation}\label{p2-AU=ilaU}
	\AA U=i\la U.
	\end{equation}Equivalently, we have the following system
	\begin{eqnarray}
	v&=&i\la u , \label{p2-f1ss}\\
	\left(S_{\tilde{b}(\cdot)}(u,\omega) \right)_x -c(\cdot)z &=&i\la v ,\label{p2-f2ss}\\
	z&=&i\la y , \label{p2-f3ss}\\
	y_{xx}+c(\cdot)v&=&i\la z , \label{p2-f4ss}\\
	-\omega_s (\cdot,s)+v&=&i\la \omega(\cdot,s) . \label{p2-f5ss}
	\end{eqnarray}
	From  \eqref{p2-reauu} and \eqref{p2-AU=ilaU}, we obtain 
\begin{equation}\label{p2-dissi=0}
		0=\Re \left(i\la U,U\right)=\Re\left(\AA U,U\right)_{\HH}=\frac{b_0}{2}\intdnb \int_{0}^{\infty}g^{\prime}(s)|\omega_x (\cdot,s)|^2 dsdx .
		\end{equation}
		Thus, we have
	\begin{equation}\label{p2-3.8}
	\omega_x (\cdot,s)=0\ \ \text{in}\ \ (0,\beta )\times (0,\infty).
	\end{equation}
	From \eqref{p2-3.8}, we have
	\begin{equation}\label{p2-3.9}
	\omega(\cdot,s)=k(s)\ \ \text{in}\ \ (x,s)\in (0,\beta)\times (0,\infty),
	\end{equation}
	where $k(s)$ is a constant depending on $s$. Then, from \eqref{p2-3.9} and the fact that $\omega(\cdot,s)\in \mathcal{W}_g $ $\left(\text{i.e.} \  \omega(0,s)=0\right)$, we get 
	\begin{equation}\label{p2-3.10}
	\omega(\cdot,s)=0\ \ \text{in}\ \ (0,\beta )\times (0,\infty).
	\end{equation}
	From \eqref{p2-f1ss}, \eqref{p2-f5ss} and the fact that $\omega(\cdot,0)=0$, we deduce that
	\begin{equation}\label{p2-3.12p}
	\omega(\cdot,s)=u(e^{-i\la s}-1), \ \ \text{in}\ \ (0,L)\times(0,\infty).
	\end{equation}
	From \eqref{p2-f1ss}, \eqref{p2-f5ss} and \eqref{p2-3.10}, we obtain
	\begin{equation}\label{p2-3.14}
	u=v=0 \ \ \text{in}\ \ (0 ,\beta).
	\end{equation}Inserting \eqref{p2-f1ss} and \eqref{p2-f3ss} in \eqref{p2-f2ss} and \eqref{p2-f4ss}, then using \eqref{p2-3.8} together with the definition of \ref{p2-S} and \ref{p2-b}, we obtain the following system 
	\begin{eqnarray}
	\la^2 u+(\widetilde{b}(\cdot)u_x )_x -c(\cdot)i\la y &=&0, \ \ \text{in}\ \ (0,L),\label{p2-3.15}\\
	\la^2 y+y_{xx} +c(\cdot)i\la u &=&0, \ \ \text{in}\ \ (0,L),\label{p2-3.16}\\
	u(0)=u(L)=y(0)=y(L)&=&0.
	\end{eqnarray}From \eqref{p2-3.14}, \eqref{p2-3.15}, the definition of \ref{p2-c} and \eqref{p2-f3ss}, we obtain 
	\begin{equation}\label{p2-3.17}
	y=z=0\ \ \text{in}\ \ (\alpha,\beta ).
	\end{equation}Thus, from \eqref{p2-3.10},  \eqref{p2-3.14} and \eqref{p2-3.17}, we obtain 
	\begin{equation}\label{p2-3.19pppp}
	U=0\ \ \text{in} \ \ (\alpha ,\beta).
	\end{equation}
Now, from \eqref{p2-3.17} and the fact that $y\in C^1 ([0,L])$, we get 
	\begin{equation}\label{p2-3.19pp}
	y(\alpha)=y_x (\alpha)=0.
	\end{equation}Next, from \eqref{p2-3.16}, \eqref{p2-3.19pp} and the definition of \ref{p2-c}, we obtain the following system 
	\begin{eqnarray}
	\la^2 y+y_{xx}&=&0, \ \ \text{in}\ \ (0,\alpha),\\
	y(0)=	y(\alpha )=y_x (\alpha )&=&0.
	\end{eqnarray}Thus, from the above system  and by using Holmgren uniqueness theorem, we obtain 
	\begin{equation}\label{p2-3.26p}
	y=0 \ \ \text{in}\ \ (0,\alpha).
	\end{equation}
	Therefore, from  \eqref{p2-f3ss}, \eqref{p2-3.12p}, \eqref{p2-3.14} and \eqref{p2-3.26p}, we obtain 
	\begin{equation}\label{p2-3.30p}
	U=0 \ \ \text{in}\ \ (0,\alpha).
	\end{equation}
		According to the definition of \ref{p2-S} and \ref{p2-btilde}, we obtain
		\begin{equation}\label{p2-3.19p}
		S_{\tilde{b}(\cdot)}(u,\omega)= au_x-b(\cdot)\widetilde{g}u_x +b(\cdot)\int_{0}^{\infty}g(s)\omega_x (\cdot,s)ds 
		\end{equation}
		From \eqref{p2-3.8}, \eqref{p2-3.14}, \eqref{p2-3.19p} and the definition of \ref{p2-b}, we get 
		\begin{equation}\label{p2-3.19pr}
		S_{\tilde{b}(\cdot)}(u,\omega)= au_x\ \ \text{in}\ (0,L) \  \text{and consequently}\ \		\left(S_{\tilde{b}(\cdot)}(u,\omega)\right)_x= au_{xx} \ \text{in}\ (0,L).
		\end{equation}Thus, from \eqref{p2-3.19pr} and the fact that $U \in D(\AA)$, we obtain 
		\begin{equation}\label{p2-3.20p}
		u_{xx}\in L^2 (0,L) \  \ \text{and consequently} \ \ u \in C^1 ([0,L]).
		\end{equation}
	Now, from \eqref{p2-3.14}, \eqref{p2-3.17}, \eqref{p2-3.20p} and the fact that $y\in C^1 ([0,L])$, we obtain 
	\begin{equation}
	u(\beta )=u_x (\beta)=	y(\beta )=y_x (\beta )=0.
	\end{equation}  
	Next, from the definition of \ref{p2-btilde} and \ref{p2-c}, the System \eqref{p2-3.15}-\eqref{p2-3.16} can be written in $(\beta ,\gamma )$ as the following system 
	\begin{eqnarray}
	\la^2 u+au_{xx} -c_0 i\la y &=&0, \ \ \text{in}\ \ (\beta,\gamma ),\label{p2-3.19}\\
	\la^2 y+y_{xx} +c_0 i\la u &=&0, \ \ \text{in}\ \ (\beta,\gamma ),\label{p2-3.20}\\
	u(\beta )=u_x (\beta )=	y(\beta )=y_x (\beta )&=&0\label{p2-3.21}.
	\end{eqnarray}
	Let $V=(u,u_x ,y,y_x )^{\top}$, then  system \eqref{p2-3.19}-\eqref{p2-3.21} can be written as the following \begin{equation}\label{p2-de}
	V_x =B V,\ \ V(\beta)=0.
	\end{equation}where
	\begin{equation*}
	B =  \begin{pmatrix}
	0&1&0&0\\
	-a^{-1}\la^2 &0&a^{-1}i\la c_0 &0\\
	0&0&0&1\\
	-i\la c_0 &0&-\la^2 &0
	\end{pmatrix}.
	\end{equation*}The solution of the differential equation \eqref{p2-de} is given by
	\begin{equation}\label{p2-solde}
	V(x)=e^{B (x-\beta )}V (\beta ),
	\end{equation}
	Thus, from \eqref{p2-solde} and the fact that $V(\beta )=0$, we get 
	\begin{equation}\label{p2-3.26}
	V=0\ \ \text{in}\ \ (\beta ,\gamma ) \ \ \text{and consequently} \ \ u=u_x =y=y_x =0 \ \ \text{in}\ \ (\beta,\gamma).
	\end{equation}
	So, from \eqref{p2-3.12p}
	and 
	\eqref{p2-3.26}, we get
	\begin{equation}\label{p2-3.27}
	U=0 \ \ \text{in}\ \ (\beta,\gamma).
	\end{equation}
	Now, from \eqref{p2-3.26} and the fact that $u,y \in C^1 ([0,L])$, we obtain
	\begin{equation}\label{p2-3.28}
	u(\gamma )=u_x (\gamma)=y(\gamma)=y_x (\gamma)=0.
	\end{equation}
	Next, from the definition of \ref{p2-btilde} and \ref{p2-c}, the system \eqref{p2-3.15}-\eqref{p2-3.16} can be written in $(\gamma ,L)$ as the following system
	\begin{eqnarray}
	\la^2 u+au_{xx}&=&0, \ \ \text{in}\ \ (\gamma ,L),\label{p2-3.29}\\
	\la^2 y+y_{xx}  &=&0, \ \ \text{in}\ \ (\gamma,L),\label{p2-3.30}\\
	u(L)=u(\gamma )=u_x (\gamma)&=&0,\\
	y(L)=y(\gamma)=y_x (\gamma)&=&0.
	\end{eqnarray}From the above system  and by using Holmgren uniqueness theorem, we deduce that 
	\begin{equation}\label{p2-3.33}
	u=y=0 \ \ \text{in}\ \ (\gamma ,L).
	\end{equation}
	Thus, from \eqref{p2-f1ss}, \eqref{p2-f3ss}, \eqref{p2-3.12p} and \eqref{p2-3.33}, we obtain 
	\begin{equation}\label{p2-3.45}
	U=0 \ \ \text{in}\ \ (\gamma ,L).
	\end{equation}Finally, from \eqref{p2-3.19pppp} \eqref{p2-3.30p}, \eqref{p2-3.27} and \eqref{p2-3.45}, we obtain
	\begin{equation}
	U=0 \ \ \text{in}\ \ (0,L).
	\end{equation}The proof is thus complete.
\end{proof}
\begin{lem}\label{p2-surj}
	{\rm Under the hypotheses $\eqref{paper2-H}$, for all $\la \in \R $, we have 
		$$R(i\la I-\AA )=\HH.$$
		
	}
\end{lem}
\begin{proof}
	From Proposition \ref{p2-mdissip}, we have $0\in\rho(\AA)$. We still need to show the result for $\la \in \R^{\star}$. For this aim, let $F=(f^1,f^2,f^3,f^4,f^5 (\cdot,s))^{\top}\in \HH$, we want to find $U=(u,v,y,z,\omega(\cdot,s))^{\top}\in D(\AA)$ solution of \begin{equation}\label{p2-2.78}
	(	i\la I -\AA)U=F.
	\end{equation}Equivalently, we have the following system 
	\begin{eqnarray}
	i\la u	-v&=&f^1 , \label{p2-f1sss}\\
	i\la v	-\left(S_{\tilde{b}(\cdot)} \right)_x +c(\cdot)z &=&f^2 ,\label{p2-f2sss}\\
	i\la y	-z&=&f^3 , \label{p2-f3sss}\\
	i\la z	-y_{xx}-c(\cdot)v&=&f^4 , \label{p2-f4sss}\\
	i\la \omega (\cdot,s)+	\omega_s (\cdot,s)-v&=&f^5(\cdot,s) , \label{p2-f5sss}
	\end{eqnarray}
	with the following boundary conditions \begin{equation}\label{p2-boundary1mdissips}
	u(0)=u(L)=y(0)=y(L)=0,\ \ \omega (\cdot,0)=0\ \ \text{in}\ \ (0,L) \ \ \text{and} \ \ \omega(0,s)=0  \ \ \text{in} \ \ (0,\infty).
	\end{equation}
	From \eqref{p2-f1sss}, \eqref{p2-f5sss} and \eqref{p2-boundary1mdissips}, we have 
	\begin{equation}\label{p2-omegas}
	\omega(\cdot,s)=\frac{1}{i\la}(i\la u-f^1 )(1-e^{-i\la s})+\int_{0}^{s}f^5 (\cdot,\xi)e^{i\la (\xi-s)}d\xi, \ \ (x,s)\in (0,L)\times (0,\infty).
	\end{equation}
	See the definition of \ref{p2-S}, inserting \eqref{p2-f1sss}, \eqref{p2-f3sss} and \eqref{p2-omegas} in \eqref{p2-f2sss} and \eqref{p2-f4sss}, we obtain the following system 
	{\small	\begin{equation}\label{p2-syssurj}
		\left\{\begin{array}{lll}
		\displaystyle	-\la^2 u -\left[\widehat{b}(\cdot)u_x +\frac{1}{i\la }b(\cdot)\int_{0}^{\infty}g(s)(1-e^{-i\la s})f^1_x ds +b(\cdot)\int_{0}^{\infty}g(s)\int_{0}^{s}f^5_x (\cdot,\xi)e^{i\la(\xi-s)}d\xi ds \right]_x +i\la c(\cdot)y=F_1, \vspace{0.15cm}\\ \displaystyle -\la^2 y-y_{xx}-i\la c(\cdot)u=F_2, \vspace{0.15cm}\\ u(0)=u(L)=y(0)=y(L)=0,\vspace{0.15cm}
		\end{array}\right.
		\end{equation}}where \begin{equation}\displaystyle\widehat{b}(\cdot)=a-b(\cdot)\int_{0}^{\infty}g(s)e^{-i\la s}ds,\ \ F_1 =f^2+ c(\cdot)f^3 +i\la f^1\ \  \text{and} \ \ F_2 =f^4 -c(\cdot)f^1 +i\la f^3 .\end{equation}
	Let $(\phi ,\psi) \in H^{1}_0 (0,L) \times H^{1}_0 (0,L)$. Multiplying the first equation of  \eqref{p2-syssurj} and the second equation of \eqref{p2-syssurj} by $\overline{\phi}$ and $\overline{\psi}$ respectively, integrating over $(0,L)$, then using integrations by parts, we obtain 
	\begin{equation}\label{p2-3.59}
	\begin{array}{lll}
	&&\displaystyle	-\la^2 \intdx u \overline{\phi}dx +\intdx\widehat{b}(\cdot)u_x \overline{\phi_x} dx +\frac{b_0 }{i\la }\intdnb\int_{0}^{\infty}g(s)(1-e^{-i\la s})f^1_x \overline{\phi_x}  dsdx \vspace{0.25cm}\\ &&\displaystyle +\,b_0 \intdnb\int_{0}^{\infty}g(s)\int_{0}^{s}e^{i\la(\xi-s)}f^5_x   (\cdot,\xi)\overline{\phi_x} d\xi dsdx  +i\la c_0 \intdnc y\overline{\phi}dx=\intdx F_1 \overline{\phi}dx
	\end{array}
	\end{equation} and 
	\begin{equation}\label{p2-3.60}
	-\la^2 \intdx y\overline{\psi}dx +\intdx y_x \overline{\psi_x} dx -i\la c_0 \intdnc u\overline{\psi}dx =\intdx F_2 \overline{\psi}dx.
	\end{equation}Adding  \eqref{p2-3.59} and \eqref{p2-3.60}, we get \begin{equation}\label{p2-2.90}
	\mathcal{B}((u,y),(\phi,\psi))=\mathcal{L}(\phi,\psi), \quad\forall (\phi,\psi)\in \mathbb{V}:=H^{1}_0 (0,L)\times  H^{1}_0 (0,L), 
	\end{equation}where  $$
	\mathcal{B}((u,y),(\phi,\psi))=\mathcal{B}_1 ((u,y),(\phi,\psi))+\mathcal{B}_2 ((u,y),(\phi,\psi))$$
	with \begin{equation}\label{p2-3.62}
	\left\{\begin{array}{lll}
	\displaystyle \mathcal{B}_1 ((u,y),(\phi,\psi))=\intdx \widehat{b}(\cdot)u_x \overline{\phi_x} dx +\intdx y_x \overline{\psi_x}dx , \vspace{0.25cm}\\\displaystyle\mathcal{B}_2 ((u,y),(\phi,\psi))=-\la^2 \intdx (u\overline{\phi}+y\overline{\psi})dx -i\la c_0 \intdx (u\overline{\psi}-y\overline{\phi})dx
	\end{array}\right.
	\end{equation} and{\small \begin{equation*}
		\begin{array}{lll}
		\displaystyle \mathcal{L}(\phi,\psi)=\displaystyle\intdx (F_1 \overline{\phi}+F_2 \overline{\psi})dx -\frac{b_0}{i\la }\intdnb \int_{0}^{\infty}g(s)(1-e^{-i\la s})f^1_x \overline{\phi_x} dsdx-b_0 \intdnb\int_{0}^{\infty}g(s)\left( \int_{0}^{s}e^{i\la (\xi-s)}f^5_x (\cdot,\xi)d\xi\right)\overline{\phi_x}dsdx. \vspace{0.25cm}
		\end{array}
		\end{equation*}}Let $\mathbb{V}^{\prime}$  be the dual space of $\mathbb{V}$. Let us define the following operators 
	\begin{equation}
	\begin{array}{lll}
	\mathbb{B}:\ \ \mathbb{V}&\longmapsto&\mathbb{V}^{\prime} \\
	\	\ \ \ (u,y)&\longmapsto& \mathbb{B}(u,y)
	\end{array} \ \ \text{and}\ \ 
	\begin{array}{lll}
	\mathbb{B}_i :\ \ \mathbb{V}&\longmapsto&\mathbb{V}^{\prime} \\
	\	\ \ \ (u,y)&\longmapsto& \mathbb{B}_i (u,y)
	\end{array},\ \ i\in \{1,2\},
	\end{equation}such that 
	\begin{equation}\label{p2-3.63}
	\left\{	\begin{array}{lll}
	\displaystyle	(\mathbb{B}(u,y))(\phi,\psi)=\mathcal{B}((u,y),(\phi,\psi)),\ \ \forall (\phi,\psi) \in \mathbb{V},\vspace{0.15cm}\\\displaystyle 	(\mathbb{B}_i (u,y))(\phi,\psi)=\mathcal{B}_i ((u,y),(\phi,\psi)),\ \ \forall (\phi,\psi) \in \mathbb{V},\ i\in \{1,2\}.
	\end{array}\right.
	\end{equation}We need to prove that the operator $\mathbb{B}$ is an isomorphism. For this aim, we divide the proof into three steps:\\\linebreak
	\textbf{Step 1.} In this step, we want to prove that the operator $\mathbb{B}_1 $ is an isomorphism. For this aim, it is easy to see that $\mathcal{B}_1 $ is sesquilinear, continuous and coercive form on $\mathbb{V}$. Then, from \eqref{p2-3.63} and Lax-Milgram theorem, the operator $\mathbb{B}_1 $ is an isomorphism.\\\linebreak
	\textbf{Step 2.} In this step, we want to prove that the operator $\mathbb{B}_2 $ is compact. For this aim, from \eqref{p2-3.62} and \eqref{p2-3.63} we have 
	\begin{equation}
		|\mathcal{B}_2 ((u,y),(\phi,\psi))|\lesssim \|(u,y)\|_{\left(L^2(0,L)\right)^2}\|(\phi,\psi)\|_{\left(L^2(0,L)\right)^2},
	\end{equation} 
	and consequently, using the compact embedding from $\mathbb{V}$ into $\left(L^2 (0,L)\right)^2 $, we deduce that $\mathbb{B}_2 $ is a compact operator.
	 \\\linebreak
	Therefore, from the above steps, we obtain that the operator $\mathbb{B}=\mathbb{B}_1 +\mathbb{B}_2 $ is a Fredholm operator of index zero. Now, following Fredholm alternative, we still need to prove that the operator $\mathbb{B}$ is injective to obtain that the operator $\mathbb{B}$ is an isomorphism.\\\linebreak
	\textbf{Step 3.} In this step, we want to prove that the operator $\mathbb{B}$ is injective (i.e. $\ker(\mathbb{B})=\{0\}$). For this aim, let $(\widetilde{u},\widetilde{y})\in \ker(\mathbb{B})$ which gives 
	\begin{equation*}
	\mathcal{B}((\widetilde{u},\widetilde{y}),(\phi,\psi))=0,\ \ \forall (\phi,\psi)\in \mathbb{V}.
	\end{equation*}Equivalently, we have 
	\begin{equation*}
	 \intdx \widehat{b}(\cdot)\widetilde{u}_x \overline{\phi_x} dx +\intdx \widetilde{y}_x \overline{\psi_x} dx -\la^2 \intdx (\widetilde{u}\overline{\phi}+\widetilde{y}\overline{\psi})dx -i\la  \intdx  c(\cdot) (\widetilde{u}\overline{\psi}-\widetilde{y}\overline{\phi})dx=0, \ \ \forall (\phi,\psi)\in \mathbb{V}.
	\end{equation*}
	Thus, we find that 
	
			\begin{equation*}
		\left\{\begin{array}{rrr}
		\displaystyle	-\la^2 \widetilde{u} -(\widehat{b}(\cdot)\widetilde{u}_x)_x +i\la c(\cdot)\widetilde{y}=0, \vspace{0.15cm}\\ \displaystyle -\la^2 \widetilde{y}-\widetilde{y}_{xx}-i\la c(\cdot)\widetilde{u}=0, \vspace{0.15cm}\\ \widetilde{u}(0)=\widetilde{u}(L)=\widetilde{y}(0)=\widetilde{y}(L)=0.\vspace{0.15cm}
		\end{array}\right.
		\end{equation*}Therefore, the vector $\widetilde{U}$ defined by $$\widetilde{U}=(\widetilde{u},i\la\widetilde{u},\widetilde{y},i\la \widetilde{y},(1-e^{-i\la s})\widetilde{u})^{\top}$$ belongs to $D(\AA )$  and we have $$i\la \widetilde{U}-\AA\widetilde{U}=0,$$and consequently $\widetilde{U}\in \ker(i\la I-\AA)$. Then, according to Lemma \ref{p2-ker}, we obtain $\widetilde{U}=0$ and consequently $\widetilde{u}=\widetilde{y}=0$ and $\ker(\mathbb{B})=\{0\}$.\\\linebreak
		Finally, from Step 3 and Fredholm alternative, we deduce that the operator $\mathbb{B}$ is isomorphism. It is easy to see that the  operator $\mathcal{L}$ is a linear and continuous form on $\mathbb{V}$. Consequently,  \eqref{p2-2.90} admits a unique solution $(u,y)\in \mathbb{V} $. By using the classical elliptic regularity, we deduce that  $U \in D(\AA) $ is a unique solution of \eqref{p2-2.78}. The proof is thus complete. 
\end{proof}\\\linebreak
\textbf{Proof of Theorem \ref{p2-strongthm2}.} From Lemma \ref{p2-ker}, we obtain the the operator $\AA $ has no pure imaginary eigenvalues (i.e. $\sigma_p (\AA)\cap i\R=\emptyset$). Moreover, from Lemma \ref{p2-surj} and with the help of the closed graph theorem of Banach, we deduce that $\sigma(\AA )\cap i\R=\emptyset$. Therefore, according to Theorem \ref{App-Theorem-A.2}, we get that the C$_0 $-semigroup $(e^{t\AA})_{t\geq0}$ is strongly stable. The proof is thus complete. \xqed{$\square$}

\begin{rk}
{\rm We mention \cite{Akil-Timo} for a direct approach of the strong stability of Timoshenko  system in the absence of compactness of the resolvent.}
\end{rk}

\section{Exponential and Polynomial Stability}\label{Exp-Pol}
\noindent In this section, under the hypotheses \eqref{paper2-H}, we show the influence of the ratio of the wave propagation speed on the stability of system \eqref{p2-sysorg0}-\eqref{p2-initialcon}. Our main result in this part is the following  theorems.
\begin{theoreme}\label{exp-a=1}{\rm
Under the hypotheses \eqref{paper2-H}, if $a=1$, then the $C_0-$semigroup $e^{t\AA}$ is exponentially stable; i.e. there exists constants $M\geq 1$ and $\epsilon>0$ independent of $U_{0}$ such that 
$$
\|e^{t\AA}U_{0}\|_{\HH}\leq Me^{-\epsilon t}\|U_{0}\|_{\HH}
$$}
\end{theoreme}
\begin{theoreme}\label{pol-an1}{\rm
Under the  hypotheses \eqref{paper2-H}, if $a\neq 1$, then there exists $C>0$ such that for every $U_{0}\in D(\AA)$, we have 
$$
E(t)\leq \frac{C}{t}\|U_0\|^2_{D(\AA)},\quad t>0.
$$}
\end{theoreme}
\noindent Since $i\R\subset \rho(\AA)$ (see Section \ref{p2-sec3}), according to Theorem \ref{Caract} and Theorem \ref{bt}, to proof  Theorem \ref{exp-a=1} and Theorem \ref{pol-an1}, we still need to prove the following condition  
\begin{equation}\tag{${\rm H_1}$}\label{H1-cond}
\sup_{\la\in \R}\left\|\left(i\la I-\AA\right)^{-1}\right\|_{\mathcal{L}(\HH)}=O\left(\abs{\la}^{\ell}\right), \quad \text{with} \quad \ell=0 \ \ \text{or}\ \ \ell=2.
\end{equation}
We will prove condition \eqref{H1-cond} by a contradiction argument. For this purpose,
suppose that \eqref{H1-cond} is false, then there exists $\left\{(\la^n,U^n:=(u^n,v^n,y^n,z^n,\omega^n(\cdot,s))^{\top})\right\}_{n\geq1}\subset \R^{\ast} \times D(\mathcal{A})$ with
\begin{equation}\label{p2-contra-pol2}
|\la^n|\to\infty \quad \text{and}\quad \|U^n\|_{\mathcal{H}}=\|(u^n,v^n,y^n,z^n,\omega^n (\cdot,s))^{\top}\|_{\HH}=1,
\end{equation}
such that  
\begin{equation}\label{p2-eq0ps}
(\la^n )^{\ell} (i\la^n I-\AA )U^n =F^n:=(f^{1,n},f^{2,n},f^{3,n},f^{4,n},f^{5,n}(\cdot,s))^{\top}  \to 0  \quad \text{in}\quad \HH.
\end{equation} 
For simplicity, we drop the index $n$. Equivalently, from \eqref{p2-eq0ps}, we have 
\begin{eqnarray}
i\la u	-v&=&\la^{-\ell}f^1 \label{p2-f1ps}\to 0 \quad\qquad \text{in}\quad H^{1}_0 (0,L),\\
i\la v	-\left(\Sbt\right)_x +c(\cdot)z &=&\la^{-\ell}f^2  \to 0 \quad\qquad \text{in}\quad L^2 (0,L),\label{p2-f2ps}\\
i\la y	-z&=&\la^{-\ell}f^3 \to 0 \quad\qquad \text{in}\quad H^{1}_0 (0,L), \label{p2-f3ps}\\
i\la z	-y_{xx}-c(\cdot)v&=&\la^{-\ell}f^4 \to 0 \quad\qquad \text{in}\quad L^2 (0,L), \label{p2-f4ps}\\
i\la \omega (\cdot,s)+	\omega_s (\cdot,s)-v&=&\la^{-\ell}f^5(\cdot,s) \to 0 \quad \text{in}\quad  \mathcal{W}_g . \label{p2-f5ps}
\end{eqnarray}
Here we will check the condition \eqref{H1-cond} by finding a contradiction with \eqref{p2-contra-pol2} by showing $\left\|U\right\|_{\HH}=o(1)$.  For clarity, we divide the proof into several Lemmas.
\begin{lem}\label{p2-1stlemps}
	{\rm Under the hypotheses \eqref{paper2-H}, the solution $U=(u,v,y,z,\omega(\cdot,s))^{\top}\in D(\AA)$ of system \eqref{p2-f1ps}-\eqref{p2-f5ps} satisfies the following estimations 
	\begin{equation}	
		-	\intdnb\int_{0}^{\infty}g^{\prime}(s)|\omega_x (\cdot,s)|^2 ds dx = o\left(|\la|^{-\ell}\right) \ \ \text{and}\ \
		\intdnb\int_{0}^{\infty}g(s)|\omega_x (\cdot,s)|^2 ds dx = o\left(|\la|^{-\ell}\right),\label{p2-4.8}\end{equation}
		\begin{equation}\label{p2-4.16}
		\intdnb |u_x |^2 dx =o(\abs{\la}^{-\ell})  \ \ \text{and}\ \ \intdnb \left|\Sbtz\right|^2 dx =o(\abs{\la}^{-\ell}).\end{equation}	}
\end{lem}
\begin{proof}
	First, taking the inner product of \eqref{p2-eq0ps} with $U$ in $\HH$ and using \eqref{p2-reauu}, we get
	\begin{equation}\label{p2-4.12}
	\displaystyle	 -\frac{b_0}{2}\intdnb \int_{0}^{\infty}g^{\prime}(s)|\omega_x (\cdot,s)|^2 dsdx =-\Re \left(\AA U,U\right)_{\HH}=\la^{-\ell}\Re  \left(F,U\right)_{\HH} \leq |\la|^{-\ell} \|F\|_{\HH}\|U\|_{\HH} .
	\end{equation}Thus, from \eqref{p2-4.12}, \eqref{paper2-H}  and the fact that $\|F\|_{\HH}=o(1)$ and $\|U\|_{\HH}=1$, we obtain the first estimation in \eqref{p2-4.8}. From hypotheses \eqref{paper2-H}, we obtain 
	\begin{equation}\label{p2-4.13}
	\intdnb\int_{0}^{\infty}g(s)|\omega_x (\cdot,s)|^2 ds dx\leq 	-\frac{1}{m}\intdnb\int_{0}^{\infty}g^{\prime}(s)|\omega_x (\cdot,s)|^2 ds dx.
	\end{equation}Then, from the first estimation in \eqref{p2-4.8} and \eqref{p2-4.13}, we obtain the second estimation in \eqref{p2-4.8}.
Next, inserting \eqref{p2-f1ps} in \eqref{p2-f5ps}, then deriving the resulting equation  with respect to $x$, we get 
	\begin{equation}\label{p2-4.18}
	i\la \omega_x (\cdot,s)+\omega_{sx}(\cdot,s)-i\la u_x = \la^{-\ell}f^5_x (\cdot,s)-\la^{-\ell}f^1_x . 
	\end{equation}Multiplying \eqref{p2-4.18} by $\la^{-1}g(s)\overline{u_x }$, integrating over $(0 ,\beta )\times (0,\infty)$, then taking the imaginary part, we obtain 
	\begin{equation*}
	\begin{array}{lll}
	\displaystyle	 \intdnb \intdm g(s)|u_x |^2 dsdx &=&	\displaystyle\Im\left\{i\intdnb\intdm g(s)\omega_x (\cdot,s)\overline{u_x }dsdx \right\} +\Im\left\{\la^{-1}\intdnb\intdm g(s)\omega_{xs}(\cdot,s)\overline{u_x }dsdx \right\}\vspace{0.25cm}\\&&	\displaystyle\,-\Im\left\{\la^{-(\ell+1)}\intdnb\intdm g(s)f^5_x (\cdot,s)\overline{u_x }dsdx \right\}+\Im\left\{\la^{-(\ell+1)}\intdnb\intdm g(s)f^1_x \overline{u_x }dsdx \right\}.
	\end{array}
	\end{equation*}Using integration by parts with respect to $s$ in the above equation, then using hypotheses \eqref{paper2-H} and the fact that $\omega(\cdot,0)=0$, we get 
	\begin{equation}\label{p2-4.19}
	\begin{array}{lll}
	\displaystyle	 \widetilde{g}\intdnb|u_x |^2 dx &=&	\displaystyle\Im\left\{i\intdnb\intdm g(s)\omega_x (\cdot,s)\overline{u_x }dsdx \right\} +\Im\left\{\la^{-1}\intdnb\intdm -g^{\prime}(s)\omega_{x}(\cdot,s)\overline{u_x }dsdx \right\}\vspace{0.25cm}\\&&	\displaystyle\,-\Im\left\{\la^{-(\ell+1)}\intdnb\intdm g(s)f^5_x (\cdot,s)\overline{u_x }dsdx \right\}+\Im\left\{\widetilde{g}\la^{-(\ell+1)}\intdnb f^1_x \overline{u_x }dx \right\}.
	\end{array}
	\end{equation}Using Young's inequality and Cauchy-Schwarz inequality in  \eqref{p2-4.19} with the help of hypotheses \eqref{paper2-H}, we obtain
	\begin{equation*}
	\begin{array}{lll}
	\displaystyle\widetilde{g}\intdnb|u_x |^2 dx &\leq&\displaystyle \frac{\widetilde{g}}{2}\intdnb|u_x|^2 dx+\frac{1}{2}\intdnb\intdm g(s)|\omega_x (\cdot,s)|^2 dsdx \vspace{0.25cm}\\ 
	&& \displaystyle +\, |\la|^{-1}\sqrt{g_0 }\left( \intdnb\intdm -g^{\prime}(s)|\omega_x (\cdot,s)|^2 dsdx\right)^{\frac{1}{2}}\left(\intdnb|u_x |^2 dx \right)^{\frac{1}{2}} \vspace{0.25cm}\\ 
	&& \displaystyle  +\, |\la|^{-(\ell+1)} \sqrt{\widetilde{g}}\left(\intdnb\intdm g(s)|f^5_x  (\cdot,s)|^2 dsdx\right)^{\frac{1}{2}}\left(\intdnb|u_x|^2 dx \right)^{\frac{1}{2}}\\
	 &&\displaystyle +\,\widetilde{g}|\la|^{-(\ell+1)}\left(\intdnb |f^1_x |^2dx \right)^{\frac{1}{2}}\left(\intdnb |u_x |^2dx \right)^{\frac{1}{2}}.
	\end{array}
	\end{equation*}From the above inequality, \eqref{p2-4.8} and the fact that $u_x $ is uniformly bounded in $L^2 (0,L )$ and $f^1_x \to 0 \ \ \text{in}\ \ L^2 (0,L)$, $f^5(\cdot,s)\to 0 \ \ \text{in}\ \ \mathcal{W}_g$, we obtain the first estimation in \eqref{p2-4.16}. Now, by using Cauchy-Schwarz inequality, we obtain 
	\begin{equation*}
	\begin{array}{lll}
	\displaystyle \intdnb \left|S_{\widetilde{b_0}}(u,\omega)\right|^2 dx= 	\intdnb \left|\widetilde{b_0 }u_x +b_0 \int_{0}^{\infty}g(s)\omega_x (\cdot,s)ds\right|^2 dx\leq\displaystyle 	2(\widetilde{b_0})^2\intdnb  |u_x|^2  +2b_0^2 \intdnb \left(\int_{0}^{\infty}g(s)|\omega_x (\cdot,s)|ds\right)^2 dx\vspace{0.25cm}\\  \hspace{9.1cm}\leq \displaystyle 2(\widetilde{b_0})^2\intdnb  |u_x|^2  +2b_0^2 \widetilde{g} \intdnb \int_{0}^{\infty}g(s)|\omega_x (\cdot,s)|^2 dsdx.
	\end{array}
	\end{equation*}
	Finally, from the above inequality,  \eqref{p2-4.8} and the first estimation in \eqref{p2-4.16}, we obtain the second estimation in \eqref{p2-4.16}. The proof is thus complete.
\end{proof}


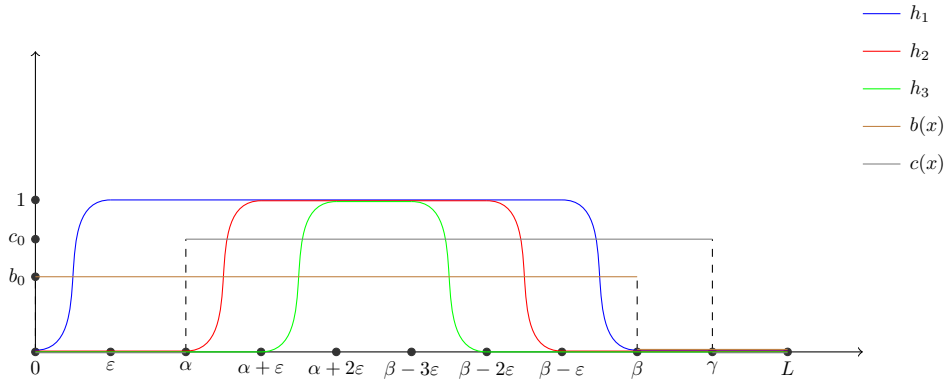
\begin{figure}[h]
\begin{center}
	\begin{tikzpicture}
\draw[->](1,0)--(12,0);
\draw[->](1,0)--(1,4);
\draw[dashed](1,0)--(1,1);
\draw[dashed](9,0)--(9,1);
\draw[dashed](3,0)--(3,1.5);
\draw[dashed](10,0)--(10,1.5);

\node[black,below] at (1,0){\scalebox{0.75}{$0$}};
\node at (1,0) [circle, scale=0.3, draw=black!80,fill=black!80] {};
\node[black,below] at (2,0){\scalebox{0.75}{$\varepsilon$}};
\node at (2,0) [circle, scale=0.3, draw=black!80,fill=black!80] {};
\node[black,below] at (3,0){\scalebox{0.75}{$\alpha $}};
\node at (3,0) [circle, scale=0.3, draw=black!80,fill=black!80] {};

\node[black,below] at (4,0){\scalebox{0.75}{$\alpha +\varepsilon$}};
\node at (4,0) [circle, scale=0.3, draw=black!80,fill=black!80] {};

\node[black,below] at (5,0){\scalebox{0.75}{$\alpha +2\varepsilon$}};
\node at (5,0) [circle, scale=0.3, draw=black!80,fill=black!80] {};

\node[black,below] at (6,0){\scalebox{0.75}{$\beta-3\varepsilon$}};
\node at (6,0) [circle, scale=0.3, draw=black!80,fill=black!80] {};

\node[black,below] at (7,0){\scalebox{0.75}{$\beta-2\varepsilon$}};
\node at (7,0) [circle, scale=0.3, draw=black!80,fill=black!80] {};

\node[black,below] at (8,0){\scalebox{0.75}{$\beta-\varepsilon$}};
\node at (8,0) [circle, scale=0.3, draw=black!80,fill=black!80] {};

\node[black,below] at (9,0){\scalebox{0.75}{$\beta$}};
\node at (9,0) [circle, scale=0.3, draw=black!80,fill=black!80] {};

\node[black,below] at (10,0){\scalebox{0.75}{$\gamma$}};
\node at (10,0) [circle, scale=0.3, draw=black!80,fill=black!80] {};

\node[black,below] at (11,0){\scalebox{0.75}{$L$}};
\node at (11,0) [circle, scale=0.3, draw=black!80,fill=black!80] {};


\node at (1,2.022) [circle, scale=0.3, draw=black!80,fill=black!80] {};
\node at (1,1) [circle, scale=0.3, draw=black!80,fill=black!80] {};
\node at (1,1.5) [circle, scale=0.3, draw=black!80,fill=black!80] {};

\node[black,left] at (1,2.022){\scalebox{0.75}{$1$}};
\node[black,left] at (1,1){\scalebox{0.75}{$b_0$}};
\node[black,left] at (1,1.5){\scalebox{0.75}{$c_0$}};

\node[black,right] at (12.5,4.5){\scalebox{0.75}{$h_1$}};
\node[black,right] at (12.5,4){\scalebox{0.75}{$h_2$}};
\node[black,right] at (12.5,3.5){\scalebox{0.75}{$h_3$}};
\node[black,right] at (12.5,3){\scalebox{0.75}{$b(x)$}};
\node[black,right] at (12.5,2.5){\scalebox{0.75}{$c(x)$}};

\draw [blue] (1,0.022) to[out=0.40,in=180] (2,2.022) ;
\draw[-,blue](2,2.022)--(8,2.022);
\draw [blue] (8,2.022) to[out=0.40,in=180] (9,0.022) ;
\draw[-,blue](9,0.022)--(11,0.022);
\draw[-,red](1,0.011)--(3,0.011);
\draw [red] (3,0.011) to[out=0.40,in=180] (4,2.011) ;
\draw[-,red](4,2.011)--(7,2.011);
\draw [red] (7,2.011) to[out=0.40,in=180] (8,0.011) ;
\draw[-,red](8,0.011)--(11,0.011);
\draw[-,green](1,0)--(4,0);
\draw [green] (4,0) to[out=0.40,in=180] (5,2) ;
\draw[-,green](5,2)--(6,2);
\draw [green] (6,2) to[out=0.40,in=180] (7,0) ;
\draw[-,green](7,0)--(11,0);

\draw[-,brown](1,1)--(9,1);
\draw[-,brown](9,0.033)--(11,0.033);

\draw[-,gray](1,-0.011)--(3,-0.011);
\draw[-,gray](3,1.5)--(10,1.5);
\draw[-,gray](10,-0.011)--(11,-0.011);

\draw[-,blue](12,4.5)--(12.5,4.5);
\draw[-,red](12,4)--(12.5,4);
\draw[-,green](12,3.5)--(12.5,3.5);
\draw[-,brown](12,3)--(12.5,3);
\draw[-,gray](12,2.5)--(12.5,2.5);
\end{tikzpicture}\end{center}
\caption{Geometric description of the functions $h_1$, $h_2$ and $h_3$.}\label{p2-Fig1}
\end{figure}
\begin{lem}\label{estimation-v}{\rm
Let $0<\varepsilon<\min\left(\alpha,\frac{\beta-\alpha}{5}\right)$. Under the hypotheses \eqref{paper2-H},  the solution $U=(u,v,y,z,\omega(\cdot,s))^{\top}\in D(\AA)$ of  \eqref{p2-f1ps}-\eqref{p2-f5ps} satisfies the following estimation}
\begin{equation}\label{estimation-v1}
\int_{\varepsilon}^{\beta-\varepsilon}\abs{v}^2dx=o\left(\abs{\la}^{-\frac{\ell}{2}}\right).
\end{equation}
\end{lem}
\begin{proof}
First, we fix a cut-off function $h_1\in C^{1}\left([0,L]\right)$ (see Figure \ref{p2-Fig1}) such that $0\leq h_1 (x)\leq 1$, for all $x\in[0,L]$ and 
\begin{equation}\tag{$h_1$}\label{p2-h1}
	h_1 (x)= 	\left \{ \begin{array}{lll}
	1 &\text{if} \quad \,\,  x \in [\varepsilon ,\beta -\varepsilon],&\vspace{0.1cm}\\
		0 &\text{if } \quad x \in \{0\}\cup [\beta ,L],&
	\end{array}	\right. \qquad\qquad
	\end{equation}
and set
\begin{equation*}
\max_{x\in [0,L]}|h'_1(x)|=M_{h_1'}.
\end{equation*}
Multiplying 	\eqref{p2-f2ps} by $\displaystyle{-h_1\int_0^{\infty}g(s)\overline{\omega}(\cdot,s)}ds$ and integrate over $(0,L)$, using integration by parts with the help of the properties of \ref{p2-h1} (i.e. $h_1 (0)=h_1 (L)=0$), then using the definition of \ref{p2-c}, we obtain 
\begin{equation}\label{estimation-v2}
\begin{array}{lll}
\displaystyle{-i\la\int_0^Lh_1v\int_0^{\infty}g(s)\overline{\omega}(\cdot,s)dsdx}=\displaystyle{\int_0^L\Sbt\left(h_1\int_0^{\infty}g(s)\overline{\omega}(\cdot,s)\right)_xdsdx}\vspace{0.25cm}\\\hspace{4cm}
\displaystyle+\, c_0\int_{\alpha}^{\beta} h_1z\int_0^{\infty}g(s)\overline{\omega}(\cdot,s)dsdx-\la^{-\ell}\int_0^Lh_1 f^2\int_0^{\infty}g(s)\overline{\omega}(\cdot,s)dsdx. 
\end{array}
\end{equation}
From \eqref{p2-f5ps}, we deduce that 
\begin{equation*}
-i\la \overline{\omega}(\cdot,s)=-\overline{\omega_s}(\cdot,s)+\overline{v}+\la^{-\ell}\overline{f^5}(\cdot,s).
\end{equation*}
Inserting the above equation in the left hand side of \eqref{estimation-v2}, then using the definition of \ref{p2-c} and \ref{p2-h1}, we get 
\begin{equation}\label{estimation-v3}
\begin{array}{lll}
\displaystyle{\widetilde{g}\int_0^Lh_1\abs{v}^2dx}=\displaystyle{\int_0^Lh_1v\int_0^{\infty}g(s)\overline{\omega_s}(\cdot,s)dsdx-\la^{-\ell}\int_0^Lh_1v\int_0^{\infty}g(s)\overline{f^5}(\cdot,s)dsdx}\vspace{0.25cm}\\\hspace{2.75cm}
\displaystyle{+\,\int_0^L\Sbt h_1'\int_0^{\infty}g(s)\overline{\omega}(\cdot,s)dsdx+\int_0^L\Sbt h_1\int_0^{\infty}g(s)\overline{\omega_x}(\cdot,s)dsdx}\vspace{0.25cm}\\\hspace{3cm}
\displaystyle{+\,c_0\int_{\alpha}^{\beta} h_1z\int_0^{\infty}g(s)\overline{\omega}(\cdot,s)dsdx-\la^{-\ell}\int_0^Lh_1 f_2\int_0^{\infty}g(s)\overline{\omega}(\cdot,s)dsdx}.
\end{array}
\end{equation}
Using  integration by parts with respect to $s$ with the help of  $\omega(\cdot,0)=0$ and hypotheses \eqref{paper2-H}, Cauchy-Schwarz inequality, Poincar\'e inequality,  $v$ is uniformly bounded in $L^2(0,L)$   and  \eqref{p2-4.16}, we get 
\begin{equation}\label{estimation-v4}
\begin{array}{lll}
\displaystyle{\left|\int_0^Lh_1v\int_0^{\infty}g(s)\overline{\omega_s}(\cdot,s)dsdx\right|}=\displaystyle{\left|\int_0^Lh_1v\int_0^{\infty}-g'(s)\overline{\omega}(\cdot,s)dsdx\right|}\vspace{0.25cm}\\\hspace{4cm}
\leq\displaystyle{\sqrt{g_0}\left(\intdnb\abs{v}^2dx\right)^{\frac{1}{2}}\left(\intdnb\int_0^{\infty}-g'(s)\abs{\omega(\cdot,s)}^2dsdx\right)^{\frac{1}{2}}}\vspace{0.25cm}\\\hspace{4cm}
\lesssim\displaystyle{\sqrt{g_0}\left(\intdnb\abs{v}^2dx\right)^{\frac{1}{2}}\left(\intdnb\int_0^{\infty}-g'(s)\abs{\omega_x (\cdot,s)}^2dsdx\right)^{\frac{1}{2}}=o\left(\abs{\la}^{-\frac{\ell}{2}}\right)}.
\end{array}
\end{equation}
Using the definition of \ref{p2-h1}, Cauchy-Schwarz inequality, Poincar\'e inequality, \eqref{p2-4.8} and the fact that $v, z$ are uniformly bounded in $L^2 (0,L)$ and $\|f^2\|_{L^2 (0,L)}=o(1)$, $f^5 \to 0 \ \text{in}\ \mathcal{W}_g$, we get
{\small\begin{equation}\label{estimation-v5}
\left\{\begin{array}{lll}
\displaystyle{\left|\la^{-\ell}\int_0^Lh_1v\int_0^{\infty}g(s)\overline{f^5}(\cdot,s)dsdx\right|\lesssim\frac{1}{\abs{\la}^{\ell}}\sqrt{\widetilde{g}}\left(\intdnb\abs{v}^2dx\right)^{\frac{1}{2}}\left(\intdnb\int_0^{\infty}g(s)\abs{f^5_x (\cdot,s)}^2dsdx\right)^{\frac{1}{2}}=\frac{o(1)}{\abs{\la}^{\ell}}},\vspace{0.25cm}\\
\displaystyle{\left|c_0 \int_{\alpha}^{\beta} h_1z\int_0^{\infty}g(s)\overline{\omega}(\cdot,s)dsdx\right|\leq c_0\sqrt{\widetilde{g}}\left(\int_{\alpha}^{\beta}\abs{z}^2dx\right)\left(\int_{\alpha}^{\beta}\int_0^{\infty}g(s)\abs{\omega (\cdot,s)}^2dsdx\right)^{\frac{1}{2}}}\vspace{0.25cm}\\
\hspace{4.5cm}\displaystyle \lesssim c_0\sqrt{\widetilde{g}}\left(\int_{\alpha}^{\beta}\abs{z}^2dx\right)\left(\int_{0}^{\beta}\int_0^{\infty}g(s)\abs{\omega_x (\cdot,s)}^2dsdx\right)^{\frac{1}{2}}=\frac{o(1)}{|\la|^{\frac{\ell}{2}}},
\vspace{0.25cm}\\
\displaystyle{\left|\la^{-\ell}\int_0^{L}h_1f_2\int_0^{\infty}g(s)\overline{\omega}(\cdot,s)dsdx\right|\lesssim\frac{1}{\abs{\la}^{\ell}}\sqrt{\widetilde{g}}\left(\intdnb\abs{f_2}^2dx\right)^{\frac{1}{2}}\left(\intdnb\int_0^{\infty}g(s)\abs{\omega_x (\cdot,s)}^2dsdx\right)^{\frac{1}{2}}=\frac{o(1)}{\abs{\la}^{\frac{3\ell}{2}}}}.
\end{array}
\right.
\end{equation}}On the other hand, we have 
\begin{equation*}
\left\{\begin{array}{l}
\displaystyle{|\Sbtz||h_1'|g(s)|\omega(\cdot,s)|\leq \frac{1}{2}\abs{h_1'}\abs{\Sbtz}^2g(s)+\frac{1}{2}\abs{h_1'}\abs{\omega(\cdot,s)}^2g(s)},\vspace{0.25cm}\\
\displaystyle{|\Sbtz||h_1|g(s)|\omega_x (\cdot,s)|\leq \frac{1}{2}\abs{h_1}\abs{\Sbtz}^2g(s)+\frac{1}{2}\abs{h_1}\abs{\omega_x (\cdot,s)}^2g(s)}.
\end{array}
\right.
\end{equation*}
 Then from the above inequalities, the definition of \ref{p2-S} and \ref{p2-h1},  Poincar\'e inequality and estimations \eqref{p2-4.8}  and \eqref{p2-4.16}, we obtain 
{\small\begin{equation}\label{estimation-v6}
\left\{\begin{array}{lll}
\displaystyle{\left|\int_0^L\Sbt h_1'\int_0^{\infty}g(s)\overline{\omega}(\cdot,s)dsdx\right|\leq \frac{M_{h_1'}}{2}\left(\widetilde{g}\intdnb\abs{\Sbtz}^2dx+C_p \intdnb\int_0^{\infty}g(s)\abs{\omega_x(\cdot,s)}^2dsdx\right)}\vspace{0.25cm}\\
\hspace{5.5cm}=o(|\la|^{-\ell}),\vspace{0.25cm}\\
\displaystyle{\left|\int_0^L\Sbt h_1\int_0^{\infty}g(s)\overline{\omega_x}(\cdot,s)dsdx\right|\leq \frac{\widetilde{g}}{2}\intdnb|\Sbtz|^2dx+\intdnb\int_0^{\infty}g(s)\abs{\omega_x(\cdot,s)}^2dsdx=o(\abs{\la}^{-\ell})},
\end{array}
\right.
\end{equation}}where $C_p >0$ is a Poincar\'e constant. Inserting inequalities \eqref{estimation-v4}-\eqref{estimation-v6} in \eqref{estimation-v3}, we obtain 
\begin{equation*}
\int_0^Lh_1\abs{v}^2dx=o\left(\abs{\la}^{-\frac{\ell}{2}}\right).
\end{equation*}
Finally, from the above estimation and the definition of \ref{p2-h1}, we obtain the desired result \eqref{estimation-v1}. The proof is thus complete.
\end{proof}
\begin{lem}\label{estimation-yx}{\rm
Let $0<\varepsilon<\min\left(\alpha,\frac{\beta-\alpha}{5}\right)$. Under the hypotheses \eqref{paper2-H}, the solution the solution $U=(u,v,y,z,\omega(\cdot,s))^{\top}\in D(\AA)$ of  \eqref{p2-f1ps}-\eqref{p2-f5ps} satisfies the following estimation}
\begin{equation}\label{estimation-yx1}
\int_{\alpha+\varepsilon}^{\beta-2\varepsilon}\abs{y_x}^2dx\leq \abs{a-1}\abs{\la}\int_{\alpha}^{\beta-\varepsilon}\abs{u_x}\abs{y_x}dx+o(1).
\end{equation}
\end{lem}
\begin{proof}
First, we fix a  cut-off function $h_2\in C^{1}\left([0,L]\right)$ (see Figure \ref{p2-Fig1}) such that $0\leq h_2(x)\leq 1$, for all $x\in [0,L]$ and  
\begin{equation}\tag{$h_2 $}\label{p2-h2}
h_2 (x)= 	\left \{ \begin{array}{lll}
0 &\text{if } \quad x \in [0,\alpha]\cup [\beta-\varepsilon ,L],&\vspace{0.1cm}\\
1 &\text{if} \quad \,\,  x \in [\alpha +\varepsilon ,\beta -2\varepsilon],&
\end{array}	\right. \qquad\qquad
\end{equation}
and set
\begin{equation*}
\max_{x\in [0,L]}|h'_2(x)|=M_{h_2'}.
\end{equation*}
From  \eqref{p2-f4ps}, $i\la^{-1}h_2\overline{y_{xx}}$ is uniformly bounded in $L^2(0,L)$. Multiplying  \eqref{p2-f2ps} by $i\la^{-1}h_2\overline{y_{xx}}$, using  integration by parts over $(0,L)$ and over $(\alpha,\beta-\varepsilon)$,  the definitions of \ref{p2-c} and \ref{p2-h2}, and using the fact that $\|f^2\|_{L^2(0,L)}=o(1)$, we get 
\begin{equation}\label{estimation-yx2}
\int_0^Lh_2'v\overline{y_x}dx+\int_0^Lh_2v_x\overline{y_x}dx-\frac{i}{\la}\int_0^L h_2 \left(\Sbt\right)_x \overline{y_{xx}}dx-\frac{ic_0}{\la}\int_{\alpha}^{\beta-\varepsilon}\left(h_2'z\overline{y_x}+h_2z_x\overline{y_x}\right)dx=\frac{o(1)}{|\la|^{\ell}}.
\end{equation}
From  \eqref{p2-f1ps} and \eqref{p2-f3ps}, we obtain 
\begin{equation*}
v_x=i\la u_x-\la^{-\ell}f^1_x\quad \text{and}\quad -\frac{i}{\la}z_x=y_x+i\la^{-(\ell+1)}f^3_x.
\end{equation*}
Inserting the above equations in  \eqref{estimation-yx2} and taking the real part, we get 
{\small\begin{equation}\label{estimation-yx3}
\begin{array}{l}
\displaystyle{c_0\int_{0}^{L}h_2\abs{y_x}^2dx+\Re\left\{i\la\int_0^Lh_2u_x\overline{y_x}dx\right\}-\Re\left\{\frac{i}{\la}\int_0^L\left(\Sbt\right)_xh_2\overline{y_{xx}}dx\right\}=-\Re\left\{\int_0^Lh_2'v\overline{y_x}dx\right\}}\vspace{0.25cm}\\
\displaystyle{+\,\Re\left\{\frac{1}{\la^{\ell}}\int_0^Lh_2f^1_x\overline{y_x}dx\right\}+\Re\left\{i\frac{c_0}{\la}\int_{\alpha}^{\beta-\varepsilon}h_2'z\overline{y_x}dx\right\}
	-\Re\left\{\frac{ic_0}{\la^{\ell+1}}\int_{\alpha}^{\beta-\varepsilon}h_2f^3_x\overline{y_x}dx\right\}+\frac{o(1)}{|\la|^{\ell}}}.
\end{array}
\end{equation}}
\noindent Using the fact that $y_x$ is uniformly bounded in $L^2(0,L)$, $\|f^1_x\|_{L^2 (0,L)}=o(1)$ and $\|f^3_x\|_{L^2 (0,L)}=o(1)$, we get 
\begin{equation}\label{estimation-yx4}
\Re\left\{\la^{-\ell}\int_0^Lh_2f^1_x\overline{y_x}dx\right\}=o(|\la|^{-\ell})\quad \text{and}\quad -\Re\left\{ic_0\la^{-(\ell+1)}\int_{\alpha}^{\beta-\varepsilon}h_2f^3_x\overline{y_x}dx\right\}=o(|\la|^{-(\ell+1)}).
\end{equation}
Using Cauchy-Schwarz inequality, the definition of  \ref{p2-h2}, $y_x$ and $z$ are uniformly bounded in $L^2(0,L)$, and estimation \eqref{estimation-v1}, we get 
\begin{equation}\label{estimation-yx5}
-\Re\left\{\int_0^Lh_2'v\overline{y_x}dx\right\}=o\left(\abs{\la}^{-\frac{\ell}{4}}\right)\quad \text{and}\quad \Re\left\{i\frac{c_0}{\la}\int_{\alpha}^{\beta-\varepsilon}h_2'z\overline{y_x}dx\right\}=O\left(\abs{\la}^{-1}\right)=o(1).
\end{equation}
Inserting  \eqref{estimation-yx4} and \eqref{estimation-yx5} in  \eqref{estimation-yx3}, then using the definition of \ref{p2-h2}, we get 
\begin{equation}\label{estimation-yx6}
\int_{\alpha}^{\beta-\varepsilon}h_2\abs{y_x}^2dx+\Re\left\{i\la\int_\alpha^{\beta-\varepsilon}h_2u_x\overline{y_x}dx\right\}-\Re\left\{\frac{i}{\la}\int_\alpha^{\beta-\varepsilon}\left(\Sbt\right)_xh_2\overline{y_{xx}}dx\right\}=o(1).
\end{equation}
From  \eqref{p2-f2ps}, $i\la^{-1}h_2\left(\Sbt \right)_x$ is uniformly bounded in $L^2(0,L)$. Multiplying  \eqref{p2-f4ps} by  $i\la^{-1}\left(\overline{S_{\tilde{b}(\cdot)}}(u,\omega)\right)_x$, using integration by parts over $(0,L)$, the definitions of \ref{p2-c}, \ref{p2-h2} and \ref{p2-S}, and the fact that $\|f^4\|_{L^2(0,L)}=o(1)$, we get 
\begin{equation}\label{estimation-yx7}
\begin{array}{lll}
\displaystyle\int_0^Lh_2'z\overline{S_{\tilde{b}(\cdot)}}(u,\omega) dx+\int_0^Lh_2z_x \overline{S_{\tilde{b}(\cdot)}}(u,\omega)dx-\frac{i}{\la}\int_0^Lh_2y_{xx}\left(\overline{S_{\tilde{b}(\cdot)}}(u,\omega)\right)_xdx\vspace{0.25cm}\\\displaystyle+\,\frac{ic_0}{\la}\int_{\alpha}^{\beta-\varepsilon}\left(h_2'v+h_2v_x\right)\overline{S_{\widetilde{b_0}}}(u,\omega) dx=o(|\la|^{-\ell}).
\end{array}
\end{equation}
From \eqref{p2-f3ps}, we have
\begin{equation*}
z_x=i\la y_x-\la^{-\ell}f^3_x.
\end{equation*}
Using the above equation and the definition of \ref{p2-S}, \ref{p2-b}, \ref{p2-btilde} and \ref{p2-h2}, we get 
\begin{equation}\label{estimation-yx8}
\begin{array}{l}
\displaystyle{\int_0^Lh_2z_x \overline{S_{\tilde{b}(\cdot)}}(u,\omega)dx=i\la \widetilde{b_0}\int_{\alpha}^{\beta-\varepsilon}h_2y_x\overline{u_x}dx+i\la b_0\int_{\alpha}^{\beta-\varepsilon}h_2y_x\int_0^{\infty}g(s)\overline{\omega_x}(\cdot,s)dsdx}\vspace{0.25cm}\\
\displaystyle{-\,\widetilde{b_0}\la^{-\ell}\int_{\alpha}^{\beta-\varepsilon}h_2f^3_x\overline{u_x}dx-\la^{-\ell}b_0\int_{\alpha}^{\beta-\varepsilon}h_2f^3_x\int_0^{\infty}g(s)\overline{\omega_x}(\cdot,s)dsdx}.
\end{array}
\end{equation}
From  \eqref{p2-f5ps}, we have 
\begin{equation*}
i\la \overline{\omega_x}(\cdot,s)=\overline{\omega_{xs}}(\cdot,s)+i\la \overline{u_x}+\la^{-\ell}\overline{f^1_x}-\la^{-\ell}\overline{f^5_x}(\cdot,s). 
\end{equation*}
From the above equation and by using integration by parts with respect to $s$, we get 
\begin{equation}\label{estimation-yx9}
\begin{array}{l}
\displaystyle{i\la b_0\int_{\alpha}^{\beta-\varepsilon}h_2y_x\int_0^{\infty}g(s)\overline{\omega_x}(\cdot,s)dsdx=b_0\int_{\alpha}^{\beta-\varepsilon}h_2y_x\int_0^{\infty}-g'(s)\overline{\omega_x}(\cdot,s)dsdx}\vspace{0.25cm}\\
\displaystyle{+\,i\la b_0\widetilde{g}\int_{\alpha}^{\beta-\varepsilon}h_2y_x\overline{u_x}dx+b_0\widetilde{g}\la^{-\ell}\int_{\alpha}^{\beta-\varepsilon}h_2y_{x}\overline{f^1_x}-b_0\la^{-\ell}\int_{\alpha}^{\beta-\varepsilon}h_2y_x\int_0^{\infty}g(s)\overline{f^5_x}(\cdot,s)dsdx}.
\end{array}
\end{equation}
Inserting \eqref{estimation-yx9} in \eqref{estimation-yx8}, then using the fact that $\widetilde{b_0}=a-b_0\widetilde{g}$, we get 
\begin{equation}\label{estimation-yx10}
\begin{array}{l}
\displaystyle{\int_0^Lh_2z_x \overline{S_{\tilde{b}(\cdot)}}(u,\omega)dx=i\la a\int_{\alpha}^{\beta-\varepsilon}h_2y_x\overline{u_x}dx+b_0\int_{\alpha}^{\beta-\varepsilon}h_2y_x\int_0^{\infty}-g'(s)\overline{\omega_x}(\cdot,s)dsdx}\vspace{0.25cm}\\
\displaystyle{+\,b_0\widetilde{g}\la^{-\ell}\int_{\alpha}^{\beta-\varepsilon}h_2y_{x}\overline{f^1_x}-b_0\la^{-\ell}\int_{\alpha}^{\beta-\varepsilon}h_2y_x\int_0^{\infty}g(s)\overline{f^5_x}(\cdot,s)dsdx}\vspace{0.25cm}\\
\displaystyle{-\,\widetilde{b_0}\la^{-\ell}\int_{\alpha}^{\beta-\varepsilon}h_2f^3_x\overline{u_x}dx-\la^{-\ell}b_0\int_{\alpha}^{\beta-\varepsilon}h_2f^3_x\int_0^{\infty}g(s)\overline{\omega_x}(\cdot,s)dsdx}.
\end{array}
\end{equation}
Using Cauchy-Schwarz inequality,  the facts that $y_x$, $u_x$ is uniformly bounded in $L^2(0,L)$, and  estimation \eqref{p2-4.8}, $\|F\|_{\HH}=o(1)$, we get  
\begin{equation*}
\left\{\begin{array}{lll}
\displaystyle{b_0\int_{\alpha}^{\beta-\varepsilon}h_2y_x\int_0^{\infty}-g'(s)\overline{\omega_x}(\cdot,s)dsdx=o\left(\abs{\la}^{-\frac{\ell}{2}}\right),\quad b_0\widetilde{g}\la^{-\ell}\int_{\alpha}^{\beta-\varepsilon}h_2y_{x}\overline{f^1_x}=o\left(\abs{\la}^{-\ell}\right),}\vspace{0.25cm}\\
\displaystyle{-\,b_0\la^{-\ell}\int_{\alpha}^{\beta-\varepsilon}h_2y_x\int_0^{\infty}g(s)\overline{f^5_x}(\cdot,s)dsdx=o\left(\abs{\la}^{-\ell}\right),\quad -\,\widetilde{b_0}\la^{-\ell}\int_{\alpha}^{\beta-\varepsilon}h_2f^3_x\overline{u_x}dx=o(\abs{\la}^{-\ell})}\ \ \text{and}\vspace{0.25cm}\\
\displaystyle{-\,\la^{-\ell}b_0\int_{\alpha}^{\beta-\varepsilon}h_2f^3_x\int_0^{\infty}g(s)\overline{\omega_x}(\cdot,s)dsdx=o\left(\abs{\la}^{-\frac{3\ell}{2}}\right)}.
\end{array}
\right.
\end{equation*}
Inserting the above estimations in  \eqref{estimation-yx10}, we get 
\begin{equation}\label{estimation-yx11}
\int_0^Lh_2z_x \overline{S_{\tilde{b}(\cdot)}}(u,\omega)dx=i\la a\int_{\alpha}^{\beta-\varepsilon}h_2y_x\overline{u_x}dx+o\left(\abs{\la}^{-\frac{\ell}{2}}\right).
\end{equation}
From  \eqref{p2-f1ps}, we have 
\begin{equation*}
i\la^{-1}v_x=-u_x-i\la^{-(\ell+1)}f^1_x.
\end{equation*}
Then from the above equation and the definition of \ref{p2-S} and \ref{p2-h2}, we get 
\begin{equation}\label{estimation-yx12}
i\frac{c_0}{\la}\int_{\alpha}^{\beta-\varepsilon}h_2v_x\overline{S_{\widetilde{b_0}}}(u,\omega) dx=-\int_{\alpha}^{\beta-\varepsilon}u_x\overline{S_{\widetilde{b_0}}}(u,\omega)dx-i\la^{-(\ell+1)}\int_{\alpha}^{\beta-\varepsilon}f^1_x\overline{S_{\widetilde{b_0}}}(u,\omega)dx.
\end{equation}
Using Cauchy-Schwarz inequality, the definition of \ref{p2-h2}, the fact that $u_x$ is uniformly bounded in $L^2(0,L)$ and $\|f^1_x\|=o(1)$, and estimation \eqref{p2-4.16}, we get  
\begin{equation*}
-\,\int_{\alpha}^{\beta-\varepsilon}u_x\overline{S_{\widetilde{b_0}}}(u,\omega)dx=o\left(\abs{\la}^{-\frac{\ell}{2}}\right)\quad \text{and}\quad -i\la^{-(\ell+1)}\int_{\alpha}^{\beta-\varepsilon}f^1_x \overline{S_{\widetilde{b_0}}}(u,\omega)dx=o\left(\abs{\la}^{-\frac{3\ell}{2}-1}\right)
\end{equation*}
Inserting the above estimations  in  \eqref{estimation-yx12}, we get 
\begin{equation}\label{estimation-yx13}
i\frac{c_0}{\la}\int_{\alpha}^{\beta-\varepsilon}h_2v_x\overline{S_{\widetilde{b_0}}}(u,\omega) dx=o\left(\abs{\la}^{-\frac{\ell}{2}}\right).
\end{equation}
Now using the definition of \ref{p2-h2} and \ref{p2-S},  \eqref{p2-4.16}, and the fact that $v$ and $z$ are uniformly bounded in $L^2(0,L)$, we get 
\begin{equation}\label{estimation-yx14}
\displaystyle{\int_0^Lh_2'z\overline{S_{\tilde{b}(\cdot)}}(u,\omega) dx=\int_{\alpha}^{\beta-\varepsilon}h_2'z\overline{S_{\widetilde{b_0}}}(u,\omega) dx=o\left(\abs{\la}^{-\frac{\ell}{2}}\right)\ \text{and}\  i\frac{c_0}{\la}\int_{\alpha}^{\beta-\varepsilon}h_2'v\overline{S_{\widetilde{b_0}}}(u,\omega) dx=o\left(\abs{\la}^{-\frac{3\ell}{2}}\right).}
\end{equation}
Inserting  \eqref{estimation-yx11}, \eqref{estimation-yx13} and \eqref{estimation-yx14} in \eqref{estimation-yx7}, using the definition of \ref{p2-h2}, then taking the real part, we get 
\begin{equation}\label{estimation-yx15}
\Re\left\{i\la a\int_{\alpha}^{\beta-\varepsilon}h_2y_x\overline{u_x}dx\right\}-\Re\left\{\frac{i}{\la}\int_{\alpha}^{\beta-\varepsilon}h_2y_{xx}\left(\overline{S_{\tilde{b}(\cdot)}}(u,\omega)\right)_xdx\right\}=o\left(\abs{\la}^{-\frac{\ell}{2}}\right).
\end{equation}
Now, adding  \eqref{estimation-yx6} and \eqref{estimation-yx15} and using the fact that $\ell\geq0$, we get 
\begin{equation*}
\int_{\alpha}^{\beta-\varepsilon}h_2\abs{y_x}^2dx=\Re\left\{i\la (a-1)\int_{\alpha}^{\beta-\varepsilon} h_2 u_x\overline{y_x}dx\right\}+o(1).
\end{equation*}
Using the definition of \ref{p2-h2} in the above equation, we get the desired estimation \eqref{estimation-yx1}. The proof is thus complete. 
\end{proof}
\begin{lem}\label{estimation-z}{\rm
Let $0<\varepsilon<\min\left(\alpha,\frac{\beta-\alpha}{5}\right)$. Under the hypotheses \eqref{paper2-H},  the solution $U=(u,v,y,z,\omega(\cdot,s))^{\top}\in D(\AA)$ of  \eqref{p2-f1ps}-\eqref{p2-f5ps} satisfies the following estimation}
\begin{equation}\label{estimation-z1}
\int_{\alpha+2\varepsilon}^{\beta-3\varepsilon}\abs{z}^2dx\leq 3 \abs{a-1}\abs{\la}\int_{\alpha}^{\beta-\varepsilon}\abs{u_x}\abs{y_x}dx+o(1).
\end{equation}
\end{lem}
\begin{proof} First, we fix a cut-off function $h_3 \in C^1 ([0,L])$ (see Figure \ref{p2-Fig1}) such that $0\leq h_3 (x)\leq 1$, for all $x\in[0,L]$ and
\begin{equation}\tag{$h_3$}\label{p2-h3}
h_3 (x)= 	\left \{ \begin{array}{lll}
0 &\text{if } \quad x \in [0,\alpha+\varepsilon]\cup [\beta-2\varepsilon ,L],&\vspace{0.1cm}\\
1 &\text{if} \quad \,\,  x \in [\alpha+2\varepsilon ,\beta -3\varepsilon],&
\end{array}	\right. \qquad\qquad
\end{equation}
and set
\begin{equation*}
\max_{x\in [0,L]}|h'_3(x)|=M_{h_3'}.
\end{equation*}
Multiplying  \eqref{p2-f4ps} by $-i\la^{-1}h_3\overline{z}$, using  integration by parts over $(0,L)$, the fact that $z$ is uniformly bounded in $L^2(0,L)$ and $\|f^4\|=o(1)$, and the definition of \ref{p2-c}, we get 
\begin{equation}\label{estimation-z2}
\int_0^Lh_3\abs{z}^2dx-\frac{i}{\la}\int_0^Lh_3'\overline{z}y_xdx-\frac{i}{\la}\int_0^Lh_3\overline{z_x}y_xdx+i\frac{c_0}{\la}\int_{\alpha+\varepsilon}^{\beta-2\varepsilon}h_3v\overline{z}dx=o\left(\abs{\la}^{-(\ell+1)}\right).
\end{equation}
From \eqref{p2-f3ps}, we have 
\begin{equation*}
-\frac{i}{\la}\overline{z_x}=-\overline{y_x}+i\la^{-(\ell+1)}\overline{f^3_x}.
\end{equation*}
Inserting the above equation in \eqref{estimation-z2}, we get 
\begin{equation}\label{estimation-z3}
\begin{array}{lll}
\displaystyle\int_0^Lh_3\abs{z}^2dx&=&\displaystyle \int_0^Lh_3\abs{y_x}^2dx-i\la^{-(\ell+1)}\int_0^Lh_3\overline{f^3_x}y_xdx\vspace{0.25cm}\\&&
\displaystyle{+\,\frac{i}{\la}\int_0^Lh_3'\overline{z}y_xdx-i\frac{c_0}{\la}\int_{\alpha+\varepsilon}^{\beta-2\varepsilon}h_3v\overline{z}dx+o\left(\abs{\la}^{-(\ell+1)}\right)}.
\end{array}
\end{equation}
Using the fact that $\|f^3_x\|_{L^2 (0,L)}=o(1)$, $y_x$ and $z$ are  uniformly bounded in $L^2(0,L)$, and the definition of \ref{p2-h3}, we get 
\begin{equation}\label{estimation-z4}
-i\la^{-(\ell+1)}\int_0^Lh_3\overline{f^3_x}y_xdx=o\left(\abs{\la}^{-(\ell+1)}\right),\ \  \frac{i}{\la}\int_0^Lh_3'\overline{z}y_xdx=o(1) \ \  \text{and}\  \ -i\frac{c_0}{\la}\int_{\alpha+\varepsilon}^{\beta-2\varepsilon}h_3v\overline{z}dx=o(1).
\end{equation}
Using  \eqref{estimation-yx1} and the definition of  \ref{p2-h3}, we get 
\begin{equation}\label{estimation-z5}
 \int_0^Lh_3\abs{y_x}^2dx \leq 3 \int_{\alpha+\varepsilon}^{\beta-2\varepsilon}\abs{y_x}^2dx \leq 3\abs{a-1}\abs{\la}\int_{\alpha}^{\beta-\varepsilon}\abs{u_x}\abs{y_x}dx+o(1).
\end{equation}
Inserting  \eqref{estimation-z4} and \eqref{estimation-z5} in  \eqref{estimation-z3} and using the definition of  \ref{p2-h3}, we get the desired estimation \eqref{estimation-z1}. The proof has been completed. 
\end{proof}
\\[0.1in]
Now, we fix a function $\chi \in C^1 ([\beta -3\varepsilon,\gamma ])$ by 
\begin{equation}\tag{$\chi$}\label{p2-chi}
	\chi (\beta -3\varepsilon)=-\chi (\gamma)=1, \ \ \text{and  \ set} \ \  \max_{x\in[\beta -3\varepsilon,\gamma]}|\chi (x)|=M_{\chi}\ \ \text{and}\ \ \max_{x\in[\beta -3\varepsilon,\gamma]}|\chi^{\prime} (x)|=M_{\chi^{\prime}}.
\end{equation}
\begin{rk}
{\rm It is easy to see the existence of $\chi(x)$. For example, we can take 
$$
\chi(x)=\frac{1}{(\gamma-\beta+3\varepsilon)^2}\left(-2 x^2+4 (\beta-3\varepsilon)x+ \gamma^2-(\beta-3\varepsilon)^2-2\gamma(\beta-3\varepsilon)\right),
$$
to get $\chi(\beta-3\varepsilon)=-\chi(\gamma)=1$, $\chi\in C^1\left([\beta-3\varepsilon,\gamma]\right)$, $M_{\chi}=1$ and $M_{\chi^{\prime}}=\frac{4}{\gamma -\beta +3\varepsilon}$. 
}\xqed{$\square$}
\end{rk}

\begin{lem}\label{p2-3rdlemma}
	{\rm Let $0<\varepsilon<\min\left(\alpha,\frac{\beta-\alpha}{5}\right)$. Under the hypotheses \eqref{paper2-H}, the solution $U=(u,v,y,z,\omega(\cdot,s))^{\top}\in D(\AA)$ of  \eqref{p2-f1ps}-\eqref{p2-f5ps} satisfies the following estimations 
		\begin{equation}\label{p2-4.32}
		|v(\gamma )|^2 +|v(\beta -3\varepsilon)|^2 = O(|\la|) \ \ \text{and}\ \  	|z(\gamma )|^2 +|z(\beta -3\varepsilon)|^2 +	|y_x (\gamma )|^2 +|y_x (\beta -3\varepsilon)|^2 = O(1).
		\end{equation}

	}
\end{lem}
\begin{proof}
	First, deriving Equation \eqref{p2-f1ps} with respect to $x$, we obtain
	\begin{equation*}\label{p2-4.33}
	i\la u_x -v_x =\la^{-\ell}f^1_x.
	\end{equation*}
	Multiplying the above equation by $2\chi \overline{v}$, integrating over $(\beta -3\varepsilon,\gamma )$, then taking the real part, we obtain 
	
	\begin{equation}\label{p2-4.3444}
	\Re\left\{2i\la \int_{\beta -3\varepsilon}^{\gamma }\chi u_x \overline{v}dx \right\}-\int_{\beta -3\varepsilon}^{\gamma}\chi (|v|^2)_x dx =\Re\left\{2 \la^{-\ell}\int_{\beta -3\varepsilon}^{\gamma }\chi f^1_x \overline{v}dx
	\right\}.
	\end{equation}Using integration by parts in  \eqref{p2-4.3444}, we obtain 
	\begin{equation}\label{p2-4.35}
	\left[-\chi  |v|^2 \right]_{\beta -3\varepsilon}^{\gamma } = -\int_{\beta -3\varepsilon}^{\gamma}\chi^{\prime}|v|^2 dx -\Re\left\{2i\la \int_{\beta -3\varepsilon}^{\gamma }\chi u_x \overline{v}dx \right\}+\Re\left\{2 \la^{-\ell}\int_{\beta -3\varepsilon}^{\gamma }\chi f^1_x \overline{v}dx
	\right\}.
	\end{equation}Using the definition of \ref{p2-chi} and Cauchy-Schwarz inequality in \eqref{p2-4.35}, we obtain 
	\begin{equation}\label{p2-4.366}
	\begin{array}{lll}
	\displaystyle	|v(\gamma )|^2 +|v(\beta -3\varepsilon)|^2 &\leq&\displaystyle  M_{\chi^{\prime}}\int_{\beta -3\varepsilon}^{\gamma}|v|^2 dx +2\abs{\la} M_{\chi}\left(\int_{\beta -3\varepsilon}^{\gamma}|u_x|^2 dx \right)^{\frac{1}{2}}\left(\int_{\beta -3\varepsilon}^{\gamma}|v|^2 dx \right)^{\frac{1}{2}}\vspace{0.25cm}\\&& \displaystyle +\, 	2\abs{\la}^{-\ell}M_{\chi}\left(\int_{\beta -3\varepsilon}^{\gamma}|f^1_x|^2  dx \right)^{\frac{1}{2}}\left(\int_{\beta -3\varepsilon}^{\gamma}|u_x|^2 dx \right)^{\frac{1}{2}}.
	\end{array}
	\end{equation}Thus, from \eqref{p2-4.366} and the fact that $u_x ,v $ are uniformly bounded in $L^2 (0,L)$ and $\|f^1_x \|_{L^2 (0,L)}=o(1)$, we obtain the first estimation in \eqref{p2-4.32}. From  \eqref{p2-f3ps}, we have 
\begin{equation*}
i\la y_x-z_x=\la^{-\ell}f^3.
\end{equation*}
Multiplying the above equation and \eqref{p2-f4ps} by $2\chi \overline{z}$ and $2\chi \overline{y_x }$ respectively, integrating over $(\beta -3\varepsilon ,\gamma )$, using the definition of \ref{p2-c}, taking the real part, then using the fact that $y_x ,z $ are uniformly bounded in $L^2 (0,L)$ and $\|f^2\|_{L^2 (0,L)}=o(1)$ and $\|f^3_x \|_{L^2 (0,L)}=o(1)$, we obtain 
\begin{equation}\label{p2-intbzbyxps}
	\Re\left\{2i\la \int_{\beta -3\varepsilon}^{\gamma }\chi y_x \overline{z}dx \right\}-\int_{\beta -3\varepsilon}^{\gamma }\chi(|z|^2 )_x dx =o(\la^{-\ell})\end{equation}and\begin{equation}\label{p2-intbzbyxps1}
	\Re\left\{2i\la \int_{\beta -3\varepsilon}^{\gamma }\chi z \overline{y_x }dx \right\}-\int_{\beta -3\varepsilon}^{\gamma }\chi (|y_x |^2 )_x dx -\Re\left\{2c_0  \int_{\beta -3\varepsilon}^{\gamma }\chi v \overline{y_x }dx \right\} =o(\la^{-\ell}).
	\end{equation}Adding  \eqref{p2-intbzbyxps} and \eqref{p2-intbzbyxps1}, then using integration by parts, we obtain 
	\begin{equation*}\label{p2-intbzbyxps3}
	\begin{array}{lll}
	\displaystyle	\left[ -\chi (|z|^2 +|y_x |^2 ) \right]_{\beta -3\varepsilon}^{\gamma } =\displaystyle -\,\int_{\beta -3\varepsilon}^{\gamma }\chi^{\prime}(|z|^2 +|y_x |^2 )dx +\Re\left\{2c_0 \int_{\beta -3\varepsilon}^{\gamma }\chi v\overline{y_x }dx \right\}+o(\la^{-\ell})
	\end{array}
	\end{equation*}Using the definition of \ref{p2-chi} and Cauchy-Schwarz inequality in the above equation, we obtain 
	\begin{equation}\label{p2-3.132}
	\begin{array}{lll}
	&&\displaystyle|z(\gamma )|^2 +|z(\beta -3\varepsilon)|^2 +|y_x (\gamma )|^2 +|y_x (\beta -3\varepsilon)|^2 \vspace{0.25cm}\\ &\leq&\displaystyle M_{\chi^{\prime}}\int_{\beta -3\varepsilon}^{\gamma} (|z|^2 +|y_x |^2 )dx+ 2c_0 M_{\chi }\left(\int_{\beta -3\varepsilon}^{\gamma} |v|^2 dx\right)^{\frac{1}{2}} \left(\int_{\beta -3\varepsilon}^{\gamma} |y_x |^2 dx\right)^{\frac{1}{2}}+o(\la^{-\ell}).
	\end{array}
	\end{equation}Finally, from \eqref{p2-3.132} and the fact that $v$, $y_x$, $z$ are uniformly bounded in $L^2 (0,L)$, we obtain the second estimation in \eqref{p2-4.32}. The proof is thus complete.
\end{proof}

\begin{lem}\label{p2-4thlemma}{\rm Let  $\theta \in C^1 ([0,L])$ be a function with $\theta(0)=\theta(L)=0$.  Under the  hypotheses \eqref{paper2-H},  the solution $U =(u ,v ,y ,z ,\omega (\cdot,s))^{\top}\in D(\AA)$ of  \eqref{p2-f1ps}-\eqref{p2-f5ps} satisfies the following estimation}
	
	\begin{equation}\label{p2-2hsn2hynxeqps}
	\begin{array}{lll}
\displaystyle 	\intdx \theta^{\prime}\left(|v |^2 +a^{-1}\left|S_{\tilde{b}(\cdot)}(u,\omega) \right|^2 +|z|^2+|y_x|^2 \right) dx+\Re\left\{2a^{-1}\intdx c(\cdot)\theta z\overline{S_{\tilde{b}(\cdot)}}(u,\omega)dx\right\}\vspace{0.25cm}\\\displaystyle -\,\Re\left\{2\intdx c(\cdot)\theta v \overline{
	y_x }dx\right\}=o\left(|\la|^{-\frac{\ell}{2}}\right).
	\end{array}
	\end{equation}
\end{lem}

\begin{proof}
	First, from  \eqref{p2-f1ps}, we deduce that 
	\begin{eqnarray}
	i\la \overline{u_x }=-\overline{v_x }-\la^{-\ell}\overline{f^{1}_x },\label{p2-1barx}
	\end{eqnarray}
	Multiplying \eqref{p2-f2ps} by $2a^{-1}\theta \overline{S_{\tilde{b}(\cdot)}}(u,\omega)$, integrating over $(0,L)$, taking the real part, then using \eqref{p2-4.16} and the fact that $\|f^2 \|_{L^2 (0,L)}=o(1)$, we get 
		\begin{equation}\label{p2-inthsn}
		\begin{array}{lll}
	\displaystyle	\Re\left\{2i\la a^{-1} \intdx \theta v\overline{S_{\tilde{b}(\cdot)}}(u,\omega) dx\right\}-a^{-1}\intdx \theta\left(\left|S_{\tilde{b}(\cdot)}(u,\omega)\right|^2 \right)_x dx  \vspace{0.25cm}\\\displaystyle +\,\Re\left\{2a^{-1}\intdx c(\cdot)\theta z\overline{S_{\tilde{b}(\cdot)}}(u,\omega)dx\right\}=o(|\la|^{-\ell}).
	\end{array}
		\end{equation}
		See the definition of \ref{p2-S}, then inserting \eqref{p2-1barx} in the first term of the above equation, we get 
	\begin{equation}\label{p2-3.77}
		\begin{array}{lll}
		\displaystyle-\,\intdx \theta\left(|v |^2 +a^{-1}\left|S_{\tilde{b}(\cdot)}(u,\omega) \right|^2 \right)_x dx+\Re\left\{2a^{-1}\intdx c(\cdot)\theta z\overline{S_{\tilde{b}(\cdot)}}(u,\omega)dx\right\}\vspace{0.25cm}\\ \displaystyle = -\,\Re\left\{2a^{-1}\widetilde{g}\intdx \theta b(\cdot)\overline{v_x }vdx \right\} -\Re\left\{2i\la a^{-1}\intdx \theta b(\cdot)v \intdm g(s)\overline{\omega_x }(\cdot,s)dsdx \right\}+o\left(|\la|^{-\ell}\right). 
		\end{array}
		\end{equation}
		From  \eqref{p2-f5ps}, we deduce that
		\begin{equation}\label{p2-4.555}
			i\la \overline{\omega_x}(\cdot,s)=\overline{\omega_{xs}}(\cdot,s)-\overline{v_x}-\la^{-\ell}\overline{f^5_x}(\cdot,s).
		\end{equation}
		Inserting \eqref{p2-4.555} in the right hand side of \eqref{p2-3.77}, then using integration by parts with respect to $s$ with the help of hypotheses \eqref{paper2-H} and the fact that $\omega(\cdot,0)=0$, we get 
			{\small\begin{equation}\label{p2-4.566}
		\begin{array}{lll}
		\displaystyle-\,\intdx \theta\left(|v |^2 +a^{-1}\left|S_{\tilde{b}(\cdot)}(u,\omega) \right|^2 \right)_x dx+\Re\left\{2a^{-1}\intdx c(\cdot)\theta z\overline{S_{\tilde{b}(\cdot)}}(u,\omega)dx\right\}\vspace{0.25cm}\\ \displaystyle = -\,\Re\left\{2a^{-1}b_0\intdnb \theta v\intdm -g^{\prime}(s)\overline{\omega_x}(\cdot,s) dsdx \right\} -\Re\left\{2 a^{-1}\la^{-\ell}b_0 \intdnb \theta v \intdm g(s)\overline{f^5_x }(\cdot,s)dsdx \right\}+o\left(|\la|^{-\ell}\right). 
		\end{array}
		\end{equation}}Using Cauchy-Schwarz inequality, the fact that $v$ is uniformly bounded in $L^2 (0,L)$, the definition of $g$ and \eqref{p2-4.8}, we obtain 
	\begin{equation}\label{p2-4.57}
\left\{	\begin{array}{lll}
	\displaystyle  -\, \Re\left\{2a^{-1}b_0\intdnb \theta v\intdm -g^{\prime}(s)\overline{\omega_x}(\cdot,s) dsdx \right\}=o\left((|\la|^{-\frac{\ell}{2}}\right),\vspace{0.25cm}\\\displaystyle \Re\left\{2 a^{-1}\la^{-\ell}b_0 \intdnb \theta v \intdm g(s)\overline{f^5_x }(\cdot,s)dsdx \right\}=o\left(|\la|^{-\ell}\right).
	\end{array}\right.
	\end{equation}
	Inserting \eqref{p2-4.57} in \eqref{p2-4.566}, then using integration by parts and the fact that $\theta(0)=\theta(L)=0$, we obtain
	\begin{equation}\label{p2-4.58}
		\intdx \theta^{\prime}\left(|v |^2 +a^{-1}\left|S_{\tilde{b}(\cdot)}(u,\omega) \right|^2 \right) dx+\Re\left\{2a^{-1}\intdx c(\cdot)\theta z\overline{S_{\tilde{b}(\cdot)}}(u,\omega)dx\right\}=o\left(|\la|^{-\frac{\ell}{2}}\right).
	\end{equation}
Next, multiplying  \eqref{p2-f4ps} by $2h\overline{y_x }$, integrating over $(0,L)$, taking the real part, then using the fact that $y_x $ is uniformly bounded in $L^2 (0,L)$ and $\|f^4 \|_{L^2 (0,L)}=o(1)$, we obtain
	\begin{equation}\label{p2-inthynx}
	\Re\left\{ 2i\la \intdx \theta z \overline{y_x  }dx\right\}-\intdx \theta(|y_x  |^2 )_x dx -\Re\left\{2\intdx c(\cdot)\theta v \overline{
		y_x }dx\right\}=o(|\la|^{-\ell}).
	\end{equation}From  \eqref{p2-f3ps}, we deduce that 
	\begin{equation}\label{p2-4.377}
	i\la	\overline{y_x }=-\overline{z_x }-\la^{-\ell}\overline{f^3_x }.
	\end{equation}Inserting \eqref{p2-4.377} in \eqref{p2-inthynx}, then using the fact that $z$ is uniformly bounded in $L^2 (0,L)$ and $\|f^3_x \|_{L^2 (0,L)}=o(1)$, we obtain 	
	\begin{equation}\label{p2-inthynxsub}
	-\intdx \theta(|z |^2 +|y_x |^2  )_x dx -\Re\left\{2\intdx c(\cdot)\theta v \overline{y_x }dx\right\}=	o(|\la|^{-\ell}).
	\end{equation}
	Using integration by parts in \eqref{p2-inthynxsub} and the fact that $\theta(0)=\theta(L)=0$, we obtain
	\begin{equation}\label{p2-inthynxsub1}
	\intdx \theta^{\prime} (|z |^2 +|y_x |^2  )_x dx -\Re\left\{2\intdx c(\cdot)\theta v \overline{
		y_x }dx\right\}=	o(|\la|^{-\ell}).
	\end{equation}
	Finally,   adding  \eqref{p2-4.58} and \eqref{p2-inthynxsub1}, we obtain the desired estimation \eqref{p2-2hsn2hynxeqps}. The proof is thus complete.
\end{proof}\\\linebreak
Let $0<\varepsilon<\min\left(\alpha,\frac{\beta-\alpha}{5}\right)$, we fix  cut-off functions $h_4 ,h_5  \in C^1 ([0,L])$ (see Figure \ref{p2-Fig2}) such that $0\leq h_4 (x)\leq 1$, $0\leq h_5 (x)\leq 1$, for all $x\in[0,L]$ and
\begin{equation*}
	h_4 (x)= 	\left \{ \begin{array}{lll}
	1 & \text{if} \quad \,\, x \in [0,\alpha +2\varepsilon],&\vspace{0.1cm}\\
	0 &\text{if } \quad x \in [\beta -3\varepsilon,L],&
	\end{array}	\right.\text{and}\quad h_5(x)= 	\left \{ \begin{array}{lll}
	0 & \text{if} \quad \,\, x \in [0,\alpha +2\varepsilon],&\vspace{0.1cm}\\
	1 &\text{if } \quad x \in [\beta -3\varepsilon,L],&
	
	\end{array}	\right. 
	\end{equation*}
\begin{figure}[h]	
\begin{center}
	\begin{tikzpicture}
\draw[->](1,0)--(8,0);
\draw[->](1,0)--(1,4);

\draw[dashed](1,0)--(1,1);
\draw[dashed](5,0)--(5,1);
\draw[dashed](2,0)--(2,1.5);
\draw[dashed](6,0)--(6,1.5);

\node[black,below] at (1,0){\scalebox{0.75}{$0$}};
\node at (1,0) [circle, scale=0.3, draw=black!80,fill=black!80] {};

\node[black,below] at (2,0){\scalebox{0.75}{$\alpha$}};
\node at (2,0) [circle, scale=0.3, draw=black!80,fill=black!80] {};

\node[black,below] at (3,0){\scalebox{0.75}{$\alpha +2\varepsilon$}};
\node at (3,0) [circle, scale=0.3, draw=black!80,fill=black!80] {};

\node[black,below] at (4,0){\scalebox{0.75}{$\beta -3\varepsilon$}};
\node at (4,0) [circle, scale=0.3, draw=black!80,fill=black!80] {};

\node[black,below] at (5,0){\scalebox{0.75}{$\beta$}};
\node at (5,0) [circle, scale=0.3, draw=black!80,fill=black!80] {};

\node[black,below] at (6,0){\scalebox{0.75}{$\gamma$}};
\node at (6,0) [circle, scale=0.3, draw=black!80,fill=black!80] {};

\node[black,below] at (7,0){\scalebox{0.75}{$L$}};
\node at (7,0) [circle, scale=0.3, draw=black!80,fill=black!80] {};


\node at (1,1) [circle, scale=0.3, draw=black!80,fill=black!80] {};
\node at (1,1.5) [circle, scale=0.3, draw=black!80,fill=black!80] {};
\node at (1,2) [circle, scale=0.3, draw=black!80,fill=black!80] {};

\node[black,left] at (1,1.5){\scalebox{0.75}{$c_0$}};
\node[black,left] at (1,1){\scalebox{0.75}{$b_0$}};
\node[black,left] at (1,2){\scalebox{0.75}{$1$}};

\node[black,right] at (8.5,4){\scalebox{0.75}{$h_4$}};
\node[black,right] at (8.5,3.5){\scalebox{0.75}{$h_5$}};
\node[black,right] at (8.5,3){\scalebox{0.75}{$b(x)$}};
\node[black,right] at (8.5,2.5){\scalebox{0.75}{$c(x)$}};

\draw[-,red](1,2)--(3,2);
\draw [red] (3,2) to[out=0.40,in=180] (4,0) ;
\draw[-,red](4,0)--(7,0);

\draw[-,blue](1,0)--(3,0);
\draw [blue] (3,0) to[out=0.40,in=180] (4,2) ;
\draw[-,blue](4,2)--(7,2);

\draw[-,green](1,1)--(5,1);
\draw[-,green](5,0.011)--(7,0.011);

\draw[-,gray](1,-0.011)--(2,-0.011);
\draw[-,gray](2,1.5)--(6,1.5);
\draw[-,gray](6,-0.011)--(7,-0.011);

\draw[-,red](8,4)--(8.5,4);
\draw[-,blue](8,3.5)--(8.5,3.5);
\draw[-,green](8,3)--(8.5,3);
\draw[-,gray](8,2.5)--(8.5,2.5);
\end{tikzpicture}
\end{center}
\caption{Geometric description of the functions $h_4$ and $h_5$.}\label{p2-Fig2}
\end{figure}
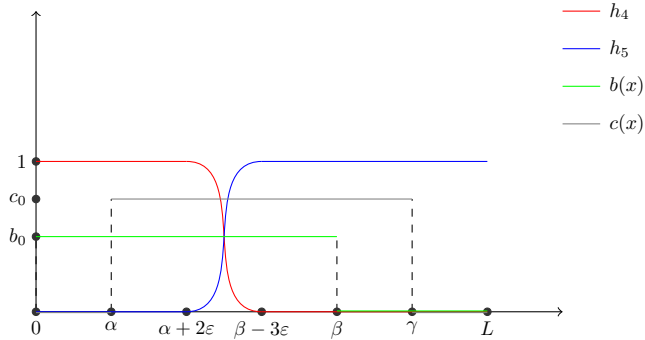
and set
$\displaystyle{\max_{x\in[0,L]}|h^{\prime}_4 (x)|=M_{h^{\prime}_4} \ \ \text{and}\ \ \max_{x\in[0,L]}|h^{\prime}_5(x)|=M_{h^{\prime}_5}.}$

\begin{lem}\label{p2-5thlemma}
	\rm{Let $0<\varepsilon<\min\left(\alpha,\frac{\beta-\alpha}{5}\right)$. Under the  hypotheses \eqref{paper2-H},  the solution $U =(u ,v ,y ,z ,\omega (\cdot, s) )^{\top}\in D(\AA)$ of  System \eqref{p2-f1ps}-\eqref{p2-f5ps} satisfies the following estimations}
	\begin{equation}\label{p2-resulycutoff3}
 \int_{0}^{\alpha +2\varepsilon}\left(|v|^2+|y_x|^2+|z|^2\right) dx \leq K_1\,|a-1||\la|\int_{\alpha}^{\beta-\varepsilon}|u_x||y_x|dx+o(1),
	\end{equation}
\begin{equation}\label{p2-resulycutoff4}
 a\int_{\beta}^{L }|u_x|^2 dx +\int_{\beta-3\varepsilon}^{L}\left( |v|^2 +|y_x|^2+|z|^2 \right)dx\leq\displaystyle K_2\,|a-1||\la|\int_{\alpha}^{\beta-\varepsilon}|u_x||y_x|dx +o(1),
\end{equation}
where $K_1=4(1+(\beta-3\varepsilon)M_{h_4'})$ and $K_2=4(1+(L-\alpha+2\varepsilon)M_{h_5'})$.
\end{lem}
\begin{proof}
First, using the result of Lemma \ref{p2-4thlemma} with $\theta =xh_4 $ and the definition of \ref{p2-S} and \ref{p2-c}, we obtain 
	\begin{equation*}
	\begin{array}{lll}
	&& \displaystyle\int_{0}^{\alpha +2\varepsilon}\left(|v|^2+|y_x|^2 + |z|^2 \right)dx=-\,a^{-1}\int_{0 }^{\alpha +2\varepsilon}\left|S_{\widetilde{b_0}}(u,\omega)\right|^{2}dx  \vspace{0.25cm}\\&&\displaystyle -\,\int_{\alpha +2\varepsilon}^{\beta -3\varepsilon}\left(h_4+x h_4^{\prime}\right) \left(|v|^2 +a^{-1}\left|S_{\widetilde{b_0}}(u,\omega)\right|^{2}+|y_x|^2+|z|^2 \right)dx\vspace{0.25cm}\\&&\displaystyle -\, \Re\left\{2a^{-1}c_0 \int_{\alpha }^{\beta -3\varepsilon}xh_4 z \overline{S_{\widetilde{b_0}}}(u,\omega)dx \right\} +\Re\left\{2c_0 \int_{\alpha}^{\beta -3\varepsilon}xh_4 v\overline{y_x}dx  \right\}+o\left(\abs{\la}^{-\frac{\ell}{2}}\right).
	\end{array}
	\end{equation*}Using Cauchy-Schwarz inequality in the above equation, we obtain 
		\begin{equation*}
			\begin{array}{lll}
	&& \displaystyle \int_{0}^{\alpha +2\varepsilon}\left(|v|^2+|y_x|^2 + |z|^2 \right)dx\leq  a^{-1}\int_{0 }^{\alpha +2\varepsilon}\left|S_{\widetilde{b_0}}(u,\omega)\right|^{2}dx  \vspace{0.25cm}\\& &\displaystyle+\, \left(1+(\beta-3\varepsilon)M_{h^{\prime}_4}\right)\int_{\alpha +2\varepsilon}^{\beta -3\varepsilon} \left(|v|^2 +a^{-1}\left|S_{\widetilde{b_0}}(u,\omega)\right|^{2}+|y_x|^2+|z|^2 \right)dx\vspace{0.25cm}\\&&\displaystyle +\,  2c_0 (\beta -3\varepsilon)a^{-1} \left(\int_{\alpha}^{\beta-3\varepsilon}|z|^2dx \right)^{\frac{1}{2}}\left(\int_{\alpha}^{\beta-3\varepsilon}|S_{\widetilde{b_0}}(u,\omega)|^2dx\right)^{\frac{1}{2}}\vspace{0.25cm}\\&&\displaystyle +\,2c_0 (\beta -3\varepsilon) \left(\int_{\alpha}^{\beta-3\varepsilon}|v|^2dx \right)^{\frac{1}{2}}\left(\int_{\alpha}^{\beta-3\varepsilon}|y_x|^2dx\right)^{\frac{1}{2}}.
	\end{array}
	\end{equation*}
Thus, from the above inequality, Lemmas \ref{p2-1stlemps}-\ref{estimation-z} and the fact that $y_x, z$  are uniformly bounded in $L^2 (0,L)$, we obtain \eqref{p2-resulycutoff3}.
	Next, using the result of Lemma \ref{p2-4thlemma} with $\theta=(x-L)h_5 $ and the definition of \ref{p2-S} and \ref{p2-c}, we obtain 
	
		\begin{equation*}
	\begin{array}{lll}
	&& \displaystyle a\int_{\beta}^{L }|u_x|^2 dx +\int_{\beta-3\varepsilon}^{L}\left( |v|^2 +|z|^2+|y_x|^2 \right)dx=-a^{-1}\int_{\beta-3\varepsilon }^{\beta }\left|S_{\widetilde{b_0}}(u,\omega)\right|^{2}dx\vspace{0.25cm}\\&&\displaystyle -\,\int_{\alpha +2\varepsilon}^{\beta -3\varepsilon}\left(h_5+(x-L) h_5^{\prime}\right) \left(|v|^2 +a^{-1}\left|S_{\widetilde{b_0}}(u,\omega)\right|^{2}+|y_x|^2+|z|^2 \right)dx\vspace{0.25cm}\\&&\displaystyle -\, \Re\left\{2a^{-1}c_0 \int_{\alpha +2\varepsilon}^{\beta -3\varepsilon}(x-L)h_5 z \overline{S_{\widetilde{b_0}}}(u,\omega)dx \right\} +\Re\left\{2c_0 \int_{\alpha+2\varepsilon}^{\beta -3\varepsilon}(x-L)h_5 v\overline{y_x}dx  \right\}\vspace{0.25cm}\\  &&\displaystyle -\, \Re\left\{2a^{-1}b_0 c_0 \int_{\beta-3\varepsilon}^{\beta }(x-L)z\left(-\widetilde{g}u_x +\intdm g(s)\overline{\omega_x}(\cdot,s)ds\right)dx  \right\}\vspace{0.25cm}\\&&\displaystyle  -\, \Re\left\{2c_0 \int_{\beta-3\varepsilon}^{\gamma}(x-L) z \overline{u_x}dx \right\} +\Re\left\{2c_0 \int_{\beta-3\varepsilon}^{\gamma }(x-L) v\overline{y_x}dx  \right\}.
	\end{array}
	\end{equation*}
	Using Cauchy-Schwarz inequality in the above equation, Lemmas \ref{p2-1stlemps}-\ref{estimation-z} and the fact that $y_x ,z$ are uniformly bounded in $L^2(0,L)$, we obtain 
\begin{equation}\label{p2-4.64}
\begin{array}{lll}
&& \displaystyle a\int_{\beta}^{L }|u_x|^2 dx +\int_{\beta-3\varepsilon}^{L}\left( |v|^2 +|z|^2+|y_x|^2 \right)dx\vspace{0.25cm}\\&\leq&\displaystyle 4\left(1+(L-\alpha -2\varepsilon)M_{h^{\prime}_5}\right)|a-1||\la|\int_{\alpha}^{\beta-\varepsilon}|u_x||y_x|dx +\mathcal{I}+o(1).
\end{array}
\end{equation}
where
\begin{equation}\label{p2-4.65}
	\mathcal{I}:= \displaystyle \Re\left\{2c_0 \int_{\beta-3\varepsilon}^{\gamma }(x-L) v\overline{y_x}dx  \right\} - \Re\left\{2c_0 \int_{\beta-3\varepsilon}^{\gamma}(x-L) z \overline{u_x}dx \right\} 
\end{equation}
From \eqref{p2-f1ps} and \eqref{p2-f3ps}, we have 
	\begin{equation}\label{p2-4.52}
	\overline{u_x }=i\la^{-1}\overline{v_x }+i\la^{-(\ell+1)}\overline{f^1_x } \ \ \text{and}\ \ \overline{y_x }=i\la^{-1}\overline{z_x }+i\la^{-(\ell+1)}\overline{f^3_x }.
	\end{equation}Inserting \eqref{p2-4.52} in \eqref{p2-4.65}, then using the fact that $v$, $z$ are uniformly bounded in $L^2 (0,L)$ and $\|f^1_x\|_{L^2 (0,L)}=o(1)$, $\|f^3_x\|_{L^2 (0,L)}=o(1)$, we obtain 
	\begin{equation}\label{p2-4.53}
	\mathcal{I}= \Re\left\{2c_0 i\la^{-1} \int_{\beta -3\varepsilon}^{\gamma } (x-L)v\overline{z_x }dx\right\}-\Re\left\{2c_0 i\la^{-1} \int_{\beta -3\varepsilon}^{\gamma } (x-L)z\overline{v_x }dx\right\}+o(|\la|^{-(\ell+1)}).
	\end{equation}Using integration by parts to  the second term in \eqref{p2-4.53}, we obtain 
	\begin{equation}\label{p2-4.54}
	\mathcal{I}=\Re\left\{2c_0 i\la^{-1} \int_{\beta -3\varepsilon}^{\gamma } z\overline{v}dx\right\}-\Re\left\{2c_0i\la^{-1}\left[(x-L)z\overline{v}\right]_{\beta -3\varepsilon}^{\gamma} \right\}+o(|\la|^{-(\ell+1)}).
	\end{equation}
	From Lemma \ref{p2-3rdlemma}, we deduce that 
	\begin{equation}\label{p2-4.56}
	|v(\gamma )|=O(\sqrt{|\la|}), \ \ 	|v(\beta -3\varepsilon )|=O(\sqrt{|\la|}),\ \ 	|z(\gamma )|=O(1) \ \ \text{and}\ \ |z(\beta-3\varepsilon )|=O(1).
	\end{equation}
	Using Cauchy-Schwarz inequality, \eqref{p2-4.56} and the fact that $v, z$ are uniformly bounded in $L^2 (0,L)$, we obtain
\begin{equation*}
\Re\left\{2c_0 i\la^{-1} \int_{\beta -3\varepsilon}^{\gamma } z\overline{v}dx\right\}=O\left(\abs{\la}^{-1}\right)=o(1)\quad \text{and}\quad -\Re\left\{2c_0i\la^{-1}\left[(x-L)z\overline{v}\right]_{\beta -3\varepsilon}^{\gamma} \right\}=O\left(|\la|^{-\frac{1}{2}}\right)=o(1).
\end{equation*}
Inserting the above estimations in \eqref{p2-4.54}, we get
\begin{equation*}
\mathcal{I}=o(1). 
\end{equation*}
Finally, from the above estimation and \eqref{p2-4.64}, we obtain the desired estimation \eqref{p2-resulycutoff4}. The proof is thus complete. 
\end{proof}
\\\linebreak
\textbf{Proof of Theorem \ref{exp-a=1}}. The proof of Theorem \ref{exp-a=1} is divided into three steps.\\
\textbf{Step 1.} Under the hypotheses \eqref{paper2-H}, by taking  $a= 1$ and $\ell=0$ in  Lemmas \ref{p2-1stlemps}-\ref{estimation-z}, we obtain 
\begin{equation}\label{1exp-a=1}
\left\{\begin{array}{l}
\displaystyle{\int_0^{\beta}\int_0^{\infty}g(s)\abs{\omega_x (\cdot,s)}^2dsdx=o(1),\ \int_0^{\beta}\abs{u_x}^2dx=o(1),\ \int_{\varepsilon}^{\beta-\varepsilon}\abs{v}^2dx=o(1),}\vspace{0.25cm}\\
\displaystyle{\int_{\alpha+\varepsilon}^{\beta-2\varepsilon}\abs{y_x}^2dx=o(1)\quad\text{and}\quad \int_{\alpha+2\varepsilon}^{\beta-3\varepsilon}\abs{z}^2dx=o(1)}.
\end{array}
\right.
\end{equation}
\textbf{Step 2.} Using the fact that $a=1$ and \eqref{1exp-a=1} in Lemma \ref{p2-5thlemma}, we obtain 
\begin{equation}\label{2exp-a=1}
\left\{\begin{array}{l}
\displaystyle
\int_0^{\varepsilon}|v|^2dx=o(1),\ \int_{\beta-\varepsilon}^{L}|v|^2dx=o(1),\ \int_{\beta}^L\abs{u_x}^2dx=o(1),\ \int_0^{\alpha+\varepsilon}\abs{y_x}^2dx=o(1),\vspace{0.25cm}\\
\displaystyle 
\int_{\beta-2\varepsilon}^L\abs{y_x}^2dx=o(1),\ \int_0^{\alpha+2\varepsilon}\abs{z}^2dx=o(1)\quad \text{and}\quad \int_{\beta-3\varepsilon}^L\abs{z}^2dx=o(1).
\end{array}
\right.
\end{equation}
\textbf{Step 3.} According to \textbf{Step 1} and \textbf{Step 2}, we obtain $\|U\|_{\mathcal{H}}=o(1)$, which contradicts \eqref{H1-cond}. Therefore, \eqref{H1-cond} holds, and so by  Theorem \ref{Caract}, we deduce that System  \eqref{p2-sysorg0}-\eqref{p2-initialcon} is exponentially stable.\xqed{$\square$}
\\\linebreak
\textbf{Proof of Theorem \ref{pol-an1}}. The proof of Theorem \ref{pol-an1} is divided into three steps.\\ 
\textbf{Step 1.} Under the hypotheses  \eqref{paper2-H} and $a\neq1$, using the fact that $y_x$ is uniformly bounded in $L^2(0,L)$ and \eqref{p2-4.16} in estimation \eqref{estimation-yx1}, we get 
\begin{equation*}
\int_{\alpha+\varepsilon}^{\beta-2\varepsilon}\abs{y_x}^2dx=o(\abs{\la}^{-\frac{\ell}{2}+1})\quad \text{and}\quad \int_{\alpha+2\varepsilon}^{\beta-3\varepsilon}\abs{z}^2dx=o(\abs{\la}^{-\frac{\ell}{2}+1}).
\end{equation*}
Taking $\ell=2$ in the above estimations,  we obtain 
\begin{equation}\label{proof-Theorem-5.2}
\int_{\alpha+\varepsilon}^{\beta-2\varepsilon}\abs{y_x}^2dx=o(1)\quad \text{and}\quad \int_{\alpha+2\varepsilon}^{\beta-3\varepsilon}\abs{z}^2dx=o(1).
\end{equation}
Taking $\ell=2$ in Lemmas   \ref{p2-1stlemps}, \ref{estimation-v}, we obtain 
\begin{equation}\label{1proof-Theorem-5.2}
\int_0^{\beta}\int_0^{\infty}g(s)\abs{\omega_x (\cdot,s)}^2dsdx=o(\la^{-2}),\ \int_0^{\beta}\abs{u_x}^2dx=o(\la^{-2})\quad \text{and}\quad \int_{\varepsilon}^{\beta-\varepsilon}\abs{v}^2dx=o(\abs{\la}^{-1}).
\end{equation}
\textbf{Step 2.} Using the fact that $a\neq 1$, $y_x$ is uniformly bounded in $L^2(0,L)$ and \eqref{1proof-Theorem-5.2} in Lemma \ref{p2-5thlemma}, we obtain 
\begin{equation}\label{1p2-resulycutoff3}
\int_{0}^{\alpha +2\varepsilon}\left(|v|^2+|y_x|^2+|z|^2\right) dx=o(1),
\end{equation}
\begin{equation}\label{1p2-resulycutoff4}
a\int_{\beta}^{L }|u_x|^2 dx +\int_{\beta-3\varepsilon}^{L}\left( |v|^2 +|y_x|^2+|z|^2 \right)dx=o(1).
\end{equation}
Using \eqref{proof-Theorem-5.2} and \eqref{1proof-Theorem-5.2} in \eqref{1p2-resulycutoff3} and \eqref{1p2-resulycutoff4}, we obtain 
\begin{equation}\label{2proof-Theorem-5.2}
\left\{\begin{array}{l}
\displaystyle
\int_0^{\varepsilon}|v|^2dx=o(1),\ \int_{\beta-\varepsilon}^{L}|v|^2dx=o(1),\ \int_{\beta}^L\abs{u_x}^2dx=o(1),\ \int_0^{\alpha+\varepsilon}\abs{y_x}^2dx=o(1),\\[0.1in]
\displaystyle 
\int_{\beta-2\varepsilon}^L\abs{y_x}^2dx=o(1),\ \int_0^{\alpha+2\varepsilon}\abs{z}^2dx=o(1)\quad \text{and}\quad \int_{\beta-3\varepsilon}^L\abs{z}^2dx=o(1).
\end{array}
\right.
\end{equation}
\textbf{Step 3.} According to \textbf{Step 1} and \textbf{Step 2}, we obtain $\|U\|_{\mathcal{H}}=o(1)$,  which contradicts \eqref{H1-cond}. This implies that 
$$
\sup_{\la\in \R}\|(i\la I-\AA)^{-1}\|_{\mathcal{\HH}}=O\left(\la^2\right).
$$ 
Finally, according to Theorem \ref{bt}, we obtain the desired result.
The proof is thus complete.\xqed{$\square$}

\section{Lack of exponential stability with global past history damping in case of different speed propagation waves ( $a\neq 1$) }\label{p2-lackexpostability}\noindent This Section is independent from the previous ones, here we prove the lack of exponential stability with global past history damping and global coupling. For this aim, we consider the following system:

\begin{equation}\label{paper2-sysorigindep} \left\{	\begin{array}{llll}\vspace{0.15cm}
\displaystyle u_{tt}-au_{xx} +\int_{0}^{\infty}g(s)u_{xx} (x,t-s)ds+y_t =0,& (x,s,t)\in (0,L)\times (0,\infty)\times (0,\infty) ,&\\ \vspace{0.15cm}
y_{tt}-y_{xx}-u_t =0,  &(x,t)\in (0,L)\times (0,\infty) ,&\\\vspace{0.15cm}
u(0,t)=u(L,t)=y(0,t)=y(L,t)=0,& t>0 ,& \\\vspace{0.15cm}
(u(x,-s),u_t (x,0))=(u_0 (x,s),u_1 (x)), &(x,s)\in (0,L)\times(0,\infty),&\\\vspace{0.15cm}	(y(x,0),y_t (x,0))=(y_0 (x),y_1 (x)), &x\in (0,L),
\end{array}\right.
\end{equation} 
\noindent and the general integral term represent a history term with the relaxation function $g$ that is supposed to satisfy the following hypotheses
\begin{equation}\tag{${\rm H_G }$}\label{paper2-HG}
\left\{ \begin{array}{lll}
g \in L^1 ([0,\infty))\cap C^1 ([0,\infty)) \ \text{be a strictly positive function such that }\vspace{0.15cm} \\\displaystyle 
\ \  g(0):=g_0 >0, \ \
\int_{0}^{\infty}g(s)ds:=\widetilde{g},\ \  
\widetilde{a}:=a-\widetilde{g}>0, \ \ \text{and}\ \ 
g^{\prime}(s)\leq -mg(s), \ \ \text{for some $m>0$}.
\end{array}\right.
\end{equation}
Now, by using the change of variable \eqref{p2-changeofvar}, then system \eqref{paper2-sysorigindep} becomes 
\begin{eqnarray}	u_{tt}-\widetilde{a}u_{xx} -\int_{0}^{\infty}g(s)\omega_{xx} (\cdot,s,t)   +y_t =0,& (x,s,t)\in (0,L)\times (0,\infty)\times (0,\infty), &\label{p2-sysorg0indep}\\ \vspace{0.15cm}
y_{tt}-y_{xx}-u_t =0,  &(x,t)\in (0,L)\times (0,\infty) ,&\label{p2-sysorg1indep}\\\vspace{0.15cm}
\omega_{t}(\cdot,s,t)+\omega_{s}(\cdot,s,t)-u_t =0,  &(x,s,t)\in (0,L)\times (0,\infty)\times (0,\infty), &\label{p2-sysorg2indep}\vspace{0.15cm}
\end{eqnarray}
with the following boundary  conditions 
\begin{equation}\label{p2-bc3indep}	\left\{\begin{array}{lll}
u(0,t)=u(L,t)=y(0,t)=y(L,t)=0,\   \ t>0, \vspace{0.15cm}\\ \omega(\cdot,0,t)=0,\ \ (x,t)\in (0,L)\times(0,\infty), \vspace{0.15cm}\\ \omega(0,s,t)=\omega(L,s,t)=0 , \ \ (s,t)\in (0,\infty)\times(0,\infty),
\end{array}\right. 
\end{equation}
and the following initial conditions
\begin{equation}\label{p2-initialconindep}
\left\{\begin{array}{llll}
u(\cdot,-s)=u_0 (\cdot,s),\qquad u_t (\cdot,0)=u_1 (\cdot), &(x,s)\in (0,L)\times(0,\infty),& \vspace{0.15cm}\\
y(\cdot,0)=y_0 (\cdot),\qquad y_t (\cdot,0)=y_1 (\cdot), &x\in (0,L),&\vspace{0.15cm} \\
\omega(\cdot,s,0)=u_0 (\cdot,0)-u_0 (\cdot,s), &(x,s)\in (0,L)\times(0,\infty).&
\end{array}\right.\end{equation} 
The energy of  system \eqref{p2-sysorg0indep}-\eqref{p2-initialconindep} is given by 
\begin{equation}
E_G (t)=\frac{1}{2}\int_{0}^{L}\left(|u_t |^2 +\widetilde{a}|u_x |^2 +|y_t |^2 +|y_x |^2 \right)dx+\frac{1 }{2}\intdx \int_{0}^{\infty}g(s)|\omega_x (\cdot,s,t) |^2 dsdx.
\end{equation}
Under the hypotheses \eqref{paper2-HG} and by letting $U=(u,v,y,z,\omega)$  be a regular solution of  system \eqref{p2-sysorg0indep}-\eqref{p2-initialconindep}, then we get with the help of \eqref{p2-bc3indep} that 
$$
\frac{d}{dt}E_G (t)=\frac{1}{2}\intdx\intdm g^{\prime}(s)|\omega_x (\cdot,s,t)|^2dsdx\leq 0,
$$
which implies that the system \eqref{p2-sysorg0indep}-\eqref{p2-initialconindep} is dissipative in the sense that its energy is non-increasing with respect to time. Now, we define the following Hilbert space $\HH_G$ by

$$
	\HH_G=\left(H^1_0 (0,L)\times L^2 (0,L)\right)^2 \times L^2_g ((0,\infty);H^1_0 (0,L)),
	$$		and it  is equipped with the following inner product \begin{equation*}\left(U,U^1 \right)_{\HH_G}=\intdx \left(\widetilde{a}u_x \overline{u_{x}^1 }+v\overline{v^1 }+y_x \overline{y_{x}^1 }+z\overline{z^1 }\right)dx +  \intdx\int_{0}^{\infty}g(s)\omega_x (\cdot,s) \overline{\omega_{x}^1 } (\cdot,s) ds dx,  \end{equation*}where  $U=(u,v,y,z,\omega(\cdot,s))^{\top}\in \HH_G$ and $U^1 =(u^1 ,v^1 ,y^1 ,z^1 ,\omega^1 (\cdot,s))^{\top}\in\HH_G$. We  define the linear unbounded  operator $\AA_G:D(\AA_G)\subset \HH_G\longmapsto \HH_G$  by:
	\begin{equation*}
	D(\AA_G)=\left\{\begin{array}{cc}\vspace{0.25cm}
	U=(u,v,y,z,\omega(\cdot,s))^{\top}\in \HH_G \,\,|\,\, y\in H^2 (0,L)\cap H^{1}_0 (0,L), \ \ v,z\in H^{1}_{0}(0,L) \\\vspace{0.25cm}
	\displaystyle  \left( \widetilde{a}u_x +\int_{0}^{\infty}g(s)\omega_x (\cdot,s)ds\right)_x  \in L^2 (0,L),\quad \omega_s (\cdot,s)\in L^2_g ((0,\infty);H^1_0 (0,L)) ,\quad \omega(\cdot,0)=0.
	\end{array}\right\} 
	\end{equation*}and 
	\begin{equation*}
	\AA_G \begin{pmatrix}
	u\\v\\y\\z\\ \omega(\cdot,s)
	\end{pmatrix}=
	\begin{pmatrix} 
	v\\\displaystyle \left( \widetilde{a}u_x +\int_{0}^{\infty}g(s)\omega_x (\cdot,s)ds\right)_x  -z\\z\\y_{xx} +v\\-\omega_s (\cdot,s) +v 
	\end{pmatrix}.
	\end{equation*} 
	Now, if $U=(u, u_t ,y, y_t ,\omega(\cdot,s))^{\top}$, then  system \eqref{p2-sysorg0indep}-\eqref{p2-initialconindep} can be written as the following first order evolution equation 
	\begin{equation}\label{p2-firstevoindep}
	U_t =\AA_G U , \quad U(0)=U_0,
	\end{equation}where  $U_0 =(u_0 (\cdot,0) ,u_1 ,y_0 ,y_1 ,\omega_0 (\cdot,s) )^{\top}\in \HH_G$.
\begin{theoreme}\label{p2-theoremexpostab}{\rm
		Under the hypotheses \eqref{paper2-HG}. If $a\neq 1$, then  for any $0<\epsilon<2$, we can not expect the energy decay rate $t^{-\frac{2}{2-\epsilon}}$ for every $U_0\in D(\AA_G)$. }
\end{theoreme}
\begin{proof}
	Following Huang \cite{Huang01} and Pruss \cite{pruss01} (see also Theorem \ref{Caract}), it is sufficient  to show the existence of sequences $\left(\la_n\right)_n\subset \R^{\ast}_+$ with $\la_n\to\infty$, $(U_n)_n\subset D(\AA_G)$ and $\left(F_n\right)_n\subset \HH_G$ such that $\left(i\la_nI-\AA\right)U_n=F_n$ is bounded in $\HH_G$ and 
	\begin{equation}\label{optimal1}
	\lim_{n\to\infty}\la_n^{-2+\epsilon}\|U_n\|_{\HH_G}=\infty. 
	\end{equation}
	For this aim, take 
	$$
	F_n=\left(0,0,0,\sin\left(\frac{n\pi x}{L}\right),0\right)\quad \text{and}\quad U_n=(u_n,i\la_nu_n,y_n,i\la_ny_n,\omega_n)
	$$
	such that 
	\begin{equation}\label{optimal2}
	\left\{\begin{array}{l}
	\displaystyle{\la_n=\frac{n\pi}{L}-\frac{L}{2n\pi (a-1)}} \ \ \text{such that}\ \ n^2>\displaystyle{\frac{L^2}{2\pi^2(a-1)}},\\[0.2in] \displaystyle{u_n(x)=A_n\sin\left(\frac{n\pi x}{L}\right)},\ \displaystyle{y_n(x)=B_n\sin\left(\frac{n\pi x}{L}\right)},\  \displaystyle{\omega_n(x,s)=A_n(1-e^{-i\la_ns})\sin\left(\frac{n\pi x}{L}\right)},
	\end{array}
	\right.
	\end{equation}
	where $A_n$ and $B_n$ are complex numbers depending on $n$ and  determined explicitly in the sequel. Note that this choice is  compatible with the boundary conditions. So, its is clear that $\la_n>0$, $\displaystyle{\lim_{n\to \infty}\la_n=\infty}$, $F_n$ is uniformly bounded in $\HH$ and $U_n\in D(\AA_G )$. Next, detailing $i\la_nU_n-\AA U_n=F_n$, we get 
	\begin{equation}\label{optimal3}
	\left\{\begin{array}{l}
	\displaystyle{iA_nL^2\la+\left(\la^2L^2-\pi^2n^2\right)B_n=-L^2},\\[0.1in]
	\displaystyle{\left(n^2\pi^2(a-g_{\la_n})-\la^2L^2\right)A_n+iL^2\la B_n=0},
	\end{array}
	\right.
	\end{equation}
	where $\displaystyle{g_{\la_n}=\int_0^{\infty}g(s)e^{-i\la_ns}ds.}$
	From the first equation of \eqref{optimal3}, we get 
	\begin{equation}\label{optimal4}
	A_n=\frac{i}{\la}+\frac{i(L^2\la^2-\pi^2n^2)B_n}{L^2\la}.
	\end{equation}
	Inserting Equation \eqref{optimal4} in the second equation of \eqref{optimal3}, we get 
	\begin{equation*}
	B_n=\frac{\left(\la^2L^2-(a-g_{\la_n})n^2\pi^2\right)L^2}{-n^4(a-g_{\la_n})\pi^4+L^2\pi^2n^2\la^2(a+1-g_{\la_n})+L^4(\la^2-\la^4)}.
	\end{equation*}
	Consequently, the solution of \eqref{optimal3} is given by 
	\begin{equation}\label{optimal5}
	A_n=\frac{i}{\la}+\frac{i(L^2\la^2-\pi^2n^2)B_n}{L^2\la}\quad \text{and}\quad B_n=B_{1,n}\left(1+\frac{B_{2,n}}{\la_n g_{\la_n}+B_{3,n}}\right),
	\end{equation}
	where
	\begin{equation*}
	\left\{\begin{array}{l}
	\displaystyle{B_{1,n}=\frac{L^2}{\left(n^2\pi^2-L^2\la^2\right)}},\quad \displaystyle{B_{2,n}=\frac{L^4\la^3}{n^2\pi^2\left(\la^2L^2-n^2\pi^2\right)}}\vspace{0.25cm}\\
	\displaystyle{B_{3,n}=\frac{\left(-\pi^4an^4+L^2n^2\la^2(a+1)\pi^2+L^4(\la^2-\la^4)\right)\la}{n^2\pi^2\left(n^2\pi^2-L^2\la^2\right)}}.
	\end{array}
	\right.
	\end{equation*}
	Now, inserting $\la_n$ given in Equation \eqref{optimal2} in the above equation, then using asymptotic expansion, we get 
	\begin{equation}\label{optimal6}
	B_{1,n}=a-1+O(n^{-2}),\quad B_{2,n}=\frac{1-a}{L}\pi n+O(n^{-1}),\quad B_{3,n}=O(n^{-1}).
	\end{equation}
	On the other hand, using hypotheses \eqref{paper2-HG} and integration by parts, we obtain
	\begin{equation*}
	\displaystyle\la_ng_{\la_n}=\displaystyle-\,ig_0 -i\int_0^{\infty}g'(s)e^{-i\la_ns}ds.
	\end{equation*}
	It is clear from  Riemann-Lebesgue Lemma that the second term in the  above equation goes to zero as $\la_n\to\infty$. 
	Thus, we obtain 
	\begin{equation}\label{optimal7}
	\la_ng_{\la_n}=-i g_0+o(1). 
	\end{equation}
	Substituting \eqref{optimal6} and \eqref{optimal7} in \eqref{optimal5}, we get 
	\begin{equation*}
	A_n=O(1)\quad \text{and}\quad B_n=\left(-\frac{i(a-1)^2}{g_0L}+o(1)\right)n\pi.
	\end{equation*}
	Therefore, from the above equation and \eqref{optimal7}, we get 
	$$
	z_n(x)=i\la_n B_n\sin\left(\frac{n\pi x}{L}\right)=\left(\frac{(a-1)^2}{g_0L^2}+o(1)\right)n^2\pi^2\sin\left(\frac{n\pi x}{L}\right).
	$$
	Consequently, 
	$$
	\left(\int_0^L\abs{z_n}^2dx\right)^{\frac{1}{2}}\sim \sqrt{\frac{L}{2}}\left(\frac{(a-1)^2}{g_0L^2}+o(1)\right)n^2\pi^2.
	$$
	Since 
	$$
	\|U_n\|_{\HH}\geq \left(\int_0^L\abs{z_n}^2dx\right)^{\frac{1}{2}}\sim  \sqrt{\frac{L}{2}}\left(\frac{(a-1)^2}{g_0L^2}+o(1)\right)n^2\pi^2\sim \la_n^2,
	$$
	then for all $0<\epsilon<2$, we have 
	$$
	\la_n^{-2+\epsilon}\|U_n\|_{\HH_1}\sim \la_n^{\epsilon}\rightarrow \infty \quad \text{as}\quad n\to\infty,
	$$
	hence, we get \eqref{optimal1}. Consequently, we cannot expect the energy decay rate $t^{-\frac{2}{2-\epsilon}}$. The proof is thus complete.
\end{proof}
\begin{rk}
	{\rm In  \cite{Almeida2011} and \cite{cordeiro}, the authors proved the lack of exponential stability of a coupled wave equations system with past history damping by taking a particular relaxation function  $g(s)=e^{-\mu s}$ such that $s\in \mathbb{R}_+$ and $\mu>1$. \xqed{$\square$}
		
	}
\end{rk}

\section{Conclusion}
We have studied  the stabilization of a locally coupled wave equations with local viscoelastic damping of past history type acting only in one equation via non smooth coefficients. We proved the strong stability of the system by using Arendt-Batty criteria. We established the exponential stability of the solution  if and only if the waves have the same speed propagation (i.e. $a=1$). In the case  $a\neq 1$, we proved that the energy of our system  decays polynomially with the rate $t^{-1}$. Lack of exponential stability result has been proved in case that the speed of waves propagation are different (i.e. $a\neq 1$). According to Theorem \ref{p2-theoremexpostab}, we can conjecture that the energy decay rate $t^{-1}$ is optimal but this question remains open.

\appendix
\section{Some notions and stability theorems}\label{p2-appendix}
\noindent In order to make this paper more self-contained, we recall in this short appendix some notions and stability results used in this work. 
\begin{defi}\label{App-Definition-A.1}{\rm
		Assume that $A$ is the generator of $C_0-$semigroup of contractions $\left(e^{tA}\right)_{t\geq0}$ on a Hilbert space $H$. The $C_0-$semigroup $\left(e^{tA}\right)_{t\geq0}$ is said to be 
		\begin{enumerate}
			\item[$(1)$] Strongly stable if 
			$$
			\lim_{t\to +\infty} \|e^{tA}x_0\|_H=0,\quad \forall\, x_0\in H.
			$$
			\item[$(2)$] Exponentially (or uniformly) stable if there exists two positive constants $M$ and $\varepsilon$ such that 
			$$
			\|e^{tA}x_0\|_{H}\leq Me^{-\varepsilon t}\|x_0\|_{H},\quad \forall\, t>0,\ \forall\, x_0\in H.
			$$
			\item[$(3)$] Polynomially stable if there exists two positive constants $C$ and $\alpha$ such that 
			$$
			\|e^{tA}x_0\|_{H}\leq Ct^{-\alpha}\|A x_0\|_{H},\quad \forall\, t>0,\ \forall\, x_0\in D(A).
			$$
			\xqed{$\square$}
	\end{enumerate}}
\end{defi}
\noindent To show  the strong stability of the $C_0$-semigroup $\left(e^{tA}\right)_{t\geq0}$ we rely on the following result due to Arendt-Batty \cite{Arendt01}. 
\begin{Theorem}\label{App-Theorem-A.2}{\rm
		{Assume that $A$ is the generator of a C$_0-$semigroup of contractions $\left(e^{tA}\right)_{t\geq0}$  on a Hilbert space $H$. If $A$ has no pure imaginary eigenvalues and  $\sigma\left(A\right)\cap i\mathbb{R}$ is countable,
			where $\sigma\left(A\right)$ denotes the spectrum of $A$, then the $C_0$-semigroup $\left(e^{tA}\right)_{t\geq0}$  is strongly stable.}\xqed{$\square$}}
\end{Theorem}
\noindent Concerning the characterisation of exponential stability of $C_0-$semigroup of contraction $(e^{tA})_{t\geq 0}$ we rely on the following result due to Huang \cite{Huang01} and Pruss \cite{pruss01}.
\begin{Theorem}\label{Caract}
	{\rm
Let $A:\ D(A)\subset H\longrightarrow H$ generates a $C_0-$semigroup of contractions $\left(e^{tA}\right)_{t\geq0}$ on $H$. Assume that $i\R \subset\rho(A)$. Then, the $C_0-$semigroup $\left(e^{tA}\right)_{t\geq0}$ is exponentially stable if and only if 
$$
\limsup_{\la\in \R,\ \abs{\la}\rightarrow +\infty}\|(i\la I-A)^{-1}\|_{\mathcal{L}(H)}<\infty.
$$}
\end{Theorem}
\noindent Also, concerning the characterization of polynomial stability stability of a $C_0-$semigroup of contraction $\left(e^{tA}\right)_{t\geq 0}$ we rely on the following result due to Borichev and Tomilov \cite{Borichev01} (see also \cite{Batty01} and \cite{RaoLiu01})
\begin{Theorem}\label{bt}
	{\rm
	Assume that $A$ is the generator of a strongly continuous semigroup of contractions $\left(e^{tA}\right)_{t\geq0}$  on $\mathcal{H}$.   If   $ i\mathbb{R}\subset \rho(\mathcal{A})$, then for a fixed $\ell>0$ the following conditions are equivalent
	\begin{equation}\label{h1}
	\sup_{\lambda\in\mathbb{R}}\left\|\left(i\lambda I-\mathcal{A}\right)^{-1}\right\|_{\mathcal{L}\left(\mathcal{H}\right)}=O\left(|\lambda|^\ell\right),
	\end{equation}
	\begin{equation}\label{h2}
	\|e^{t\mathcal{A}}U_{0}\|^2_{\HH} \leq \frac{C}{t^{\frac{2}{\ell}}}\|U_0\|^2_{D(\AA)},\hspace{0.1cm}\forall t>0,\hspace{0.1cm} U_0\in D(\AA),\hspace{0.1cm} \text{for some}\hspace{0.1cm} C>0.
	\end{equation}\xqed{$\square$}}
\end{Theorem}


\end{document}